\documentclass[10pt]{article} 

\usepackage{fullpage}
\usepackage{natbib}
\usepackage{amsmath,amssymb,amsfonts,graphicx,url,algorithm2e}

\usepackage{times}
\usepackage{dsfont}
\usepackage{mathtools}
\usepackage[utf8]{inputenc}
\usepackage[T1]{fontenc}

\usepackage{graphicx,color,fancyhdr}
\usepackage{subfigure}
\usepackage{enumerate}
\usepackage{multirow}

\usepackage{hyperref}

\newcommand{\Expect}{\mathbb{E}}
\newcommand{\ud}{\,\mathrm{d}}
\newcommand{\one}{\mathbf{1}}
\newcommand{\zero}{\mathbf{0}}

\newcommand{\op}[1]{\operatorname{#1}}
\newcommand{\norm}[1]{\|#1\|}
\newcommand{\trace}{\op{Tr}}
\newcommand{\Tan}{\op{Tan}}
\renewcommand{\tilde}{\widetilde}
\newcommand{\cvx}{\mathtt{cvx}}

\newcommand{\prim}{\op{primal}}
\newcommand{\dual}{\op{dual}}
\newcommand{\real}{\mathbb{R}}
\newcommand{\argmax}{\mathop{\operatorname{argmax}}}
\newcommand{\argmin}{\mathop{\operatorname{argmin}}}

\newcommand{\W}{\mathds{W}}

\newcommand{\calC}{\mathcal{C}}
\newcommand{\calD}{\mathcal{D}}
\newcommand{\calF}{\mathcal{F}}
\newcommand{\calG}{\mathcal{G}}
\newcommand{\calK}{\mathcal{K}}
\newcommand{\calP}{\mathcal{P}}
\newcommand{\Q}{\mathcal{Q}}
\newcommand{\calX}{\mathcal{X}}
\newcommand{\calY}{\mathcal{Y}}

\usepackage{amsthm,thmtools}

\declaretheorem[numberwithin=section]{theorem}
\declaretheorem[sibling=theorem]{lemma}
\declaretheorem[sibling=theorem]{proposition}
\declaretheorem[sibling=theorem]{corollary}
\declaretheorem[sibling=theorem]{remark}
\declaretheorem[sibling=theorem]{definition}
\declaretheorem[sibling=theorem]{example}
\declaretheorem{assumption}

\usepackage{hyperref}

\renewcommand{\equationautorefname}{Equation}
\def\equationautorefname~#1\null{%
 equation~(#1)\null
}

\numberwithin{equation}{section}

\title{2-Wasserstein Approximation via Restricted Convex Potentials\\ with Application to Improved Training for GANs}

\author{Amirhossein Taghvaei\thanks{Coordinated Science Laboratory, UIUC. Email: \texttt{taghvae2@illinois.edu}; 
research performed mostly while at Technicolor AI Lab.} \and Amin Jalali\thanks{Technicolor AI Lab. Email: \texttt{amin.jalali@technicolor.com}.}}

\begin{document}
\maketitle

\begin{abstract}
We provide a framework to approximate the 2-Wasserstein distance and the optimal transport map, amenable to efficient training as well as statistical and geometric analysis. With the quadratic cost and considering the Kantorovich dual form of the optimal transportation problem, the Brenier theorem states that the optimal potential function is convex and the optimal transport map is the gradient of the optimal potential function. Using this geometric structure, we restrict the optimization problem to different parametrized classes of convex functions and pay special attention to the class of input-convex neural networks. We analyze the statistical generalization and the discriminative power of the resulting approximate metric, and we prove a restricted moment-matching property for the approximate optimal map. Finally, we discuss a numerical algorithm to solve the restricted optimization problem and provide numerical experiments to illustrate and compare the proposed approach with the established regularization-based approaches. We further discuss practical implications of our proposal in a modular and interpretable design for GANs which connects the generator training with discriminator computations to allow for learning an overall composite generator.
\end{abstract}


\tableofcontents 
\section{Introduction}
There is a growing interest in application of the optimal transportation theory in machine learning. The main reason is that the optimal transportation theory provides a natural geometry and mathematical tools to view and manipulate probability distributions and perform optimization in this space \citep{ambrosio2008gradient}. In particular, two geometric notions, within the context of optimal transportation, are of key importance in providing these capabilities: (i) {\em a metric} to measure the similarity/discrepancy between probability distributions, i.e., the $L^p$-Wasserstein distance $\W_p(\cdot,\cdot)$ for any $p\geq 1$, and (ii) {\em a map} to transport one distribution to the other, or to interpolate between them, i.e., the optimal transport map. These geometric notions have been successfully employed in a variety of applications. Perhaps, the most well-known application is the use of the metric as a loss function in generative models to learn an underlying probability distribution, in the setting of Generative Adversarial Networks \citep{arjovsky2017wasserstein} or auto-encoders \citep{tolstikhin2017wasserstein}. The transport map is used in various applications such as in domain adaptation to adapt a learned classifier to the data from a new domain \citep{courty2017joint,courty2015optimal}, for uncoupled isotonic regression \citep{rigollet2018uncoupled}, in Bayesian inference to transport samples from the prior to the posterior distribution \citep{el2012bayesian,reich2013nonparametric}, in texture mapping for surfaces in medical imaging \citep{dominitz2010texture,rabin2011wasserstein}, in style transfer for images transferring the color distribution of one image to another \citep{ferradans2014regularized}, and in interpolating between shapes \citep{su2015optimal}, among many more applications. For a review on applications of the optimal transportation in image processing see \citet{kolouri2017optimal}. 

In the continuous settings, computing these quantities for two given distributions amounts to solving an infinite-dimensional linear program in general; see \cite{peyre2017computational} for a review of computational methods for {\em discrete} optimal transport. In order to apply the optimal transportation theory to modern machine learning tasks, that involve a large number of samples embedded in a high-dimensional space, there is a need for {\em high-quality approximations} that can be computed through fast and scalable algorithms. As evident from the uses of the distance and the map, we would like to get approximations that share the crucial properties of the exact objects. For example, as $\W_2$ is a metric and allows for optimization (e.g., see \citet{ambrosio2008gradient}), we would like its approximations to have similar metric properties and allow for optimization. The same goes for the transport map. 

Most of the existing literature, motivated by GANs, is concerned with large-scale computation and approximation of $f$-divergences \citep{nowozin2016f} or $\W_1$ \citep{arjovsky2017wasserstein}. While existing approaches are suitable for approximating the divergence in large-scale settings, they do not provide a transport map. On the other hand, there is another stream of the literature, motivated by computing the optimal transport map (or the coupling), that is based on a {\em regularized} version of the underlying optimization problem. Namely, the well-known method due to \citet{cuturi2013sinkhorn} considers the optimization problem with entropic regularization and uses the Sinkhorn iteration to solve it. 
This is a discrete method but extensions to the semi-discrete setting \citep{aurenhammer1998minkowski,levy2018notions,peyre2017computational} and the continuous setting \citep{genevay2016stochastic,seguy2017large} have also been explored. However, an inherent issue with regularized approaches is the presence of bias due to regularization. Moreover, 
changing the regularization function or reducing the regularization parameter has been observed to lead to a high iteration complexity and numerical instability, \citep{schmitzer2016stabilized,dvurechensky2018computational} and \citep[Remark 4.6]{peyre2017computational}, or a high per-iteration cost \citep{blondel2017smooth,peyre2017computational}. 

\subsection{This Paper}
In this work, we are interested in approximating the $L^2$-Wasserstein distance and the optimal transport map through changing the optimization constraints rather than through regularization. The choice of $\W_2$, rather than $\W_1$, provides us with a beautiful geometric structure; e.g., see \citet[Chapter 2]{villani2003topics}, \autoref{sec:background-OT}, or \autoref{app:background-OT}, for necessary background.

\paragraph{Methodology:} According to the well-known result due to Brenier \citep[Theorem 2.12]{villani2003topics}, the {\em optimal potential function} for the dual optimal transportation problem is {\em convex} and the optimal transport map is readily given by the {\em gradient of the optimal potential function}. Therefore, to get an approximation for the distance and the map, we propose to restrict the optimization problem to subsets of the class of convex functions. For this restriction, we propose to use the powerful class of input-convex neural network \citep{amos2016input} in practice; a neural network architecture which ensures convexity in the input variable through convex monotone activations and positive weights \citep[Section 3.2]{boyd2004convex}. We discuss our methodology in more detail in \autoref{sec:restr-cvx}. We also provide practical implications of our proposal, which lead to suggestions for practice, in \autoref{sec:practical}, and comparisons with existing strategies in \autoref{sec:prior}.

\paragraph{Theoretical Properties:}
In a study of such approximations, a natural question is the tradeoff between the {\em statistical generalization} and the {\em discriminative power} of the approximate divergence. For any parametrized subset $\calF$ of the set of convex functions, we study the approximation $\W_{2,\calF}$ as pertaining to its separating properties resulting to moment matching (\autoref{sec:moment-matching}) as well as its embedding properties compared to $\W_2$ (\autoref{sec:app-ability}). We continue this investigation from a sampling perspective, in \autoref{sec:muN-mu}, and establish statistical generalization bounds for how an empirical distribution of $N$ samples converges to the true distribution in the approximate distance. 
On the other hand, in \autoref{sec:approx-map}, we examine the approximate transport map and establish a moment-matching property when the approximate transport map is used to push one marginal forward to the other. Such a result, for example, has implications in domain adaptation, where one can design the restriction such that certain properties expressed in terms of moment statistics are preserved. Finally, in \autoref{sec:conic-summary}, we consider restriction to subsets of the set of convex functions that are convex cones and develop specialized and stronger results. Such assumption enables a duality framework discussed in full in \autoref{sec:restr-orig}.

\paragraph{Practical Implications:} 
The proposed machinery for computing approximations to $\W_2$ and the optimal transport map provides us with a number of opportunities beyond faster and large-scale computation. In \autoref{sec:homotopy}, we discuss how the parametrized approximation strategy allows for homotopy over possible parametrizations for faster training and better generalization. Moreover, in \autoref{sec:composite-wgan}, we discuss how the access to the inner-workings of the discriminator (namely the optimal map) allows for enhancing the generators learned within GANs through composition with a deterministic optimal map. For this to result in an algorithmic procedure, we discuss how our parametrized strategy provides an efficient approach. 

In \autoref{sec:numerics}, we discuss optimization strategies for solving the proposed problems and use these algorithms to compare our proposal with existing methods from a statistical-computational tradeoff point of view. We relegate all the proofs and extra expositions to appendices.

\subsection{Background on Optimal Transport Theory}\label{sec:background-OT}
The set of non-negative finite measures on $\real^d$ is denoted by $M_+(\real^d)$. 
The set of integrable functions with respect to a probability measure $\mu$ is denoted by $L^1(\mu)$. 
The set of lower-semicontinuous proper convex functions on $\real^d$ is denoted by~$\cvx(\real^d)$. 
For any $f \in \cvx(\real^d)$, the convex conjugate is given by $f^\star(y)=\sup_{x \in \real^d} \left[\langle x, y\rangle - f(x)\right]$. 
The gradient mapping for a differentiable function $f$ with respect to $x$ is denoted by $\nabla_x f(\cdot)$.

Let $\mu$ and $\nu$ be two probability distributions on $\real^d$ with finite second-order moments. 
The optimal transportation problem with quadratic cost, and its Kantorovich dual form \citep[Theorem 1.3]{villani2003topics}, are given by 
\begin{align}
\W_2^2(\mu,\nu)
\coloneqq\inf_{\pi \in \Pi(\mu,\nu)}~\int \frac{1}{2}\norm{x-y}^2_2\ud \pi(x,y)
= \sup_{(f,g)\in\calC}~\int f(x)\ud \mu(x) + \int g(y)\ud \nu(y) \label{eq:dual-form-main}
\end{align}
where $\calC \coloneqq \bigl\{(f,g)\in L^1(\mu)\times L^1(\nu):~f(x)+g(y)\leq \frac{1}{2}\norm{x-y}^2_2~\ud \mu \otimes \ud \nu~\text{a.e.}\bigr\}$, $\Pi(\mu,\nu)$ is the set of all joint measures with marginals equal to $\mu$ and $\nu$, and $\W_2(\mu,\nu)$ is the second order Wasserstein distance between $\mu$ and $\nu$. Under the change of variables $f(\cdot) \leftarrow \frac{1}{2}\norm{\cdot}_2^2-f(\cdot)$ and $g(\cdot) \leftarrow \frac{1}{2}\norm{\cdot}_2^2-g(\cdot)$, the dual problem in~\eqref{eq:dual-form-main} can be equivalently expressed as
\begin{equation}
\W_2^2(\mu,\nu)
= \int\frac{1}{2}\norm{x}_2^2 \ud \mu(x) + \int\frac{1}{2}\norm{y}_2^2 \ud \nu(y) - 
\inf_{(f,g) \in \overline{\calC} } ~\bar{J}_{\mu,\nu}(f,g) 
\label{eq:kantorovich-quad}
\end{equation}
where $\bar{J}_{\mu,\nu}(f,g) \coloneqq \int f(x)\ud \mu(x) + \int g(y)\ud \nu(y)$, and the new constraint set is
\begin{align*}
\overline{\calC} &\coloneqq \bigl\{(\bar{f},\bar{g})\in L^1(\mu)\times L^1(\nu):~\bar{f}(x)+\bar{g}(y)\geq \langle x, y\rangle~\ud \mu \otimes \ud \nu~\text{a.e.}\bigr\}.
\end{align*} 
The following result is known for the quadratic cost setting \cite[Theorems 2.9~and 2.12]{villani2003topics}.
\begin{theorem}\label{thm:Brenier-main}
	Consider the optimal transportation problem for quadratic cost function~\eqref{eq:kantorovich-quad}. Assume $\mu$ and $\nu$ have finite second order moments but do not necessarily admit a density. Then, 
	\begin{enumerate}[(i)]

		\item \label{thm:Brenier-main-1}
		There exists a pair $(f,f^\star)$, where $f \in \cvx(\real^d)$, that minimizes the dual problem in~\eqref{eq:kantorovich-quad}.

		\item \label{thm:Brenier-main-2}
(Knott-Smith optimality criterion) $\pi \in \Pi(\mu,\nu)$ is optimal for the primal problem iff there exists $f \in \cvx(\real^d)$ such that $\op{Supp}(\pi) \subset \op{Graph}(\partial f)$, or equivalently, $y \in \partial f(x)$ for all $(x,y) \in \op{Supp}(\pi)$. Moreover, the pair $(f,f^\star)$ minimizes the dual problem. 
		
		\item \label{thm:Brenier-main-3}
		(Brenier's theorem) If $\mu$ admits a density with respect to the Lebesgue measure, then the optimal coupling $\pi$ for the primal problem is unique. The optimal coupling is given by $\ud \pi(x,y) = \ud \mu(x) \delta_{y=\nabla f(x)}$ where $f \in \cvx(\real^d)$. 
	\end{enumerate}
\end{theorem}

\section{Proposed Approximation Methodology}\label{sec:restr-cvx}
Consider the setup of \autoref{sec:background-OT} and the optimization problem in~\eqref{eq:kantorovich-quad} for computing $\W_2(\mu,\nu)$; a constrained optimization problem over $\bar{\calC}$. In regularization-based approaches, as in \eqref{eq:reg-continuous}, the constraint set $\bar{\calC}$ is replaced with a penalty term. In this paper, we take a different approach and we approximate the constraint set with sets that are more computationally friendly. Moreover, parallel to the single knob of a regularization parameter, we use a family of approximations to the constraint set allowing for a richer tradeoff between the computational accuracy and efficiency; see \autoref{sec:flexible}. 

Using the Brenier theorem in \autoref{thm:Brenier-main} \cite[Theorem~2.9]{villani2003topics} we can express~\eqref{eq:kantorovich-quad} as
\begin{align}\label{eq:dual-L2-cvx}
\inf_{f,g \in \bar{\calC}}~\bar{J}_{\mu,\nu} (f,g)=\inf_{f\in \cvx(\calX)} ~\bar{J}_{\mu,\nu}(f,f^\star).
\end{align}
Then, we restrict our attention to a parametrized subset of the set of convex functions; namely $\mathcal{F}= \{f(\cdot; \theta):~ \theta\in \Theta\} \subset \cvx(\calX)$ where $\Theta \subset \real^M$ is the parameter set, and for any $\theta \in \Theta$, $f(x) = f(x; \theta)$ is a convex function in~$x$. Denote the Fenchel conjugate with respect to the first input by $f^\star(y;\theta) = \sup_x \langle x, y\rangle - f(x;\theta)$. We now solve a finite-dimensional optimization problem
\begin{align}\label{eq:min-theta}
\inf_{f\in \calF}~\bar{J}_{\mu,\nu}(f,f^\star) =\inf_{\theta \in {\Theta} }~\int f(x;\theta)\ud \mu(x) + \int f^\star(y;\theta)\ud \nu(y)
\end{align} 
where we denote the objective on the right-hand side by ${\tilde{J}}_{\mu,\nu}(\theta)$. 
Then, parallel to~\eqref{eq:kantorovich-quad}, we define the approximate metric (it is not a metric or distance but we abuse the notation here) as
\begin{align}\label{eq:W2F}
\W_{2,\calF}^2 (\mu,\nu) \coloneqq 
\int \frac{1}{2}\norm{x}_2^2 \ud \mu(x) + \int \frac{1}{2}\norm{y}_2^2 \ud \nu(y) - \inf_{\theta\in\Theta} \;\tilde{J}_{\mu,\nu}(\theta) .
\end{align}
Observe that plugging any {\em feasible} point $\theta\in\Theta$ of~\eqref{eq:min-theta} in $\tilde{J}_{\mu,\nu}(\theta)$ provides a valid upper bound for~\eqref{eq:dual-L2-cvx}, which will turn into a lower-bound for the approximate Wasserstein distance in~\eqref{eq:W2F}. Note that $\W_{2,\calF}$ is not necessarily symmetric with respect to its two arguments. Still, one can consider a symmetric version of the form $\W_{2,\calF}(\mu,\nu) + \W_{2,\calF}(\nu,\mu)$ whenever a symmetric approximation is needed. 

Finally, corresponding to the first-order optimality condition in \eqref{eq:W2F} (see \autoref{sec:approx-map} for more details), for each $\bar{\theta}\in \argmin_{\theta\in\Theta} \tilde{J}_{\mu,\nu}(\theta)$, we 
define an approximate transport map, from $\nu$ to~$\mu$, as
\begin{equation}\label{eq:approx-opt-map}
T_{\calF}(y;\bar{\theta}) \coloneqq \nabla_y f^\star(y;\bar{\theta}).
\end{equation}

\begin{remark}
The optimization problem in \eqref{eq:W2F} can be dualized to get a problem over the space of couplings whose marginals dominate $\mu$ and $\nu$ in moments specified through $\calF$. Moreover, it can be shown that strong duality holds when $\calF\subset \cvx(\calX)$ is a convex cone. We also study various metric properties for $\W_{2,\calF}$ in this case, adding to previous studies such as \cite{farnia2018convex}. 
To keep the flow of current discussion, we provide these results in \autoref{sec:restr-orig} with a concise summary in~\autoref{sec:conic-summary}.
\end{remark}
The ability to efficiently optimize $\tilde{J}_{\mu,\nu}(\theta)$ over $\theta\in\Theta$ provides us with both $\W_{2,\calF} (\mu,\nu)$ and $T_{\calF}(y;\bar{\theta}) $. A stochastic function such as $\tilde{J}_{\mu,\nu}(\theta)$ is commonly approximated through a sample average approximation. However, we still need an efficient routine to evaluate $f(\cdot; \theta)$ and an efficient representation for~$\Theta$. Moreover, we would like $\Theta$ to be expressive in parametrizing convex functions. 
Having access to an expressive enough subset of convex functions, with an efficient representation and an efficient associated method of training, allows for implementing the proposed methodology for approximation. 
As examples, one can consider classes of quadratic functions or piecewise-linear-quadratic (PLQ) functions. While these two classes enjoy very nice characterizations (see \autoref{sec:proposed-restr}), class of quadratic functions is not expressive enough, and class of PLQ functions is not easy to implement. Alternatively, we consider input-convex neural networks \citep{amos2016input} for their expressivity and their efficient way of training. Besides the discussions in \autoref{sec:proposed-restr}, we postpone the examination of other parametrized classes of convex functions \citep{helton2010semidefinite,aravkin2013sparse,ahmadi2014dsos,jalali2017variational} to future work.

\paragraph{Input-convex Neural Networks.} 
ICNNs are a class of deep neural networks whose outputs are convex with respect to their inputs. The output of the network is defined recursively according to
\begin{align*}
f(x;\theta) &= h_L \quad \text{where}\quad h_{\ell+1} = \sigma_\ell(W_\ell h_\ell + b_\ell + A_\ell x) ~,~~ \ell = 0,1,\ldots, L-1 ,
\end{align*}
where $x$ is the input, $W_\ell$ and $A_\ell$ are weight matrices (with the convention that $W_0=0$), $b_\ell$ is bias term, and $\sigma_l:\real \to \real$ is the activation function at layer $\ell$. The function $f(x;\theta)$ is convex in $x$ if (i) all the weights $W_\ell$ are positive; and (ii) the activation function $\sigma$ is non-decreasing and convex (e.g., as for ReLU activation and its pointwise square). Note that there is no constraint on weights $A_\ell$ which represent the skip connections going directly from the input to the layer $\ell$. More generally, all convexity preserving operations can be employed to define an input-convex model; e.g., see \citet[Section 3.2]{boyd2004convex} and \citet{grant2008cvx}. 

Let us discuss some factors in architecture design through comparing two architectures. With ReLU activation in all layers, the resulting ICNN will be a piecewise-linear (PL) function. In such network, changing $\sigma_1$ to a ReLU-squared results in a PLQ, and by choosing the width and depth of the network, one can adjust the complexity of the represented PLQ class. The latter architecture is preferred for our purposes for two reasons: 
(i) From a computational perspective: If in addition, the incoming weights are nonzero, e.g., by fixing them to nonzero values, the ICNN becomes strongly convex. Strong convexity allows for a more efficient computation of the Fenchel conjugate, which comes up in the inner-loop of our optimization procedure in \autoref{sec:numerics}. 
(ii) From a statistical perspective: For a PLQ, the resulting transport map in \eqref{eq:approx-opt-map} will be a piecewise affine map. Therefore, the approximation is accurate between distributions that are related to each other with a piecewise affine transformation with a limited number of pieces~(see~\autoref{sec:app-ability}) which is a rich relationship. On the other hand, with a PL function, the range of the transport map is a finite set whose size is bounded by a function of the network's width; which clearly creates generalization issues. 

In \autoref{sec:numerics}, we discuss a stochastic gradient descent procedure for solving \eqref{eq:min-theta}. Therefore, when parametrizing with an ICNN, backpropagation can be used to learn the weights where we project the updates to maintain non-negativity. 

\section{Theoretical Properties}\label{sec:main-results}
In this section, we study the approximation $\W_{2,\calF}$ and the transport map $T_{\calF}$, for any parametrized subset $\calF$ of the set of convex functions. We discuss the results of this section within the context of several restriction classes in \autoref{sec:proposed-restr}. 
We make the following assumption throughout: 
\begin{assumption}\label{assump}
(i) The marginal distributions $\mu$ and $\nu$ are supported on compact sets in $\real^d$. 
(ii) The function $f(x;\theta)$ is differentiable with respect to $\theta$ for all $x\in \mathcal{X}$ and $\nabla_\theta f(x;\theta)$ is continuous with respect to $x$. 
(iii) For all $\theta \in \Theta$, there exists $L(\theta)>0$ and a neighbourhood $U$ around $\theta$ such that $\|\nabla_\theta f(x;\theta')\|_2<L(\theta)$ for all $x\in \mathcal{X}$ and $\theta' \in U$. 
\end{assumption}
For example, 
\autoref{assump}-(ii,iii) holds for ICNNs with differentiable activation functions; e.g., squared ReLU or softplus $\log(1+e^x)$. 

\subsection{Restricted Moment-matching} \label{sec:moment-matching}
Two distributions $\mu$ and $\nu$ are said to have the moment-matching property with respect to a class of functions $\calF$, denoted by $\mu \equiv_\calF \nu$, if $\int f \ud \mu = \int f \ud \nu$ for all $f \in \calF$. This property is important in several signal processing applications when one is interested in operations that preserve certain statistics of the signal \citep{rabin2011removing,rabin2011wasserstein}. It was recently shown that the moment-matching property is achieved through GANs \citep{liu2017approximation,zhang2017discrimination}. In particular, minimizing $\W_{1,\calF}$, as defined in \eqref{eq:L1-restricted}, yields $\mu \equiv_\calF \nu$. Also see \citet{han2018local}. Using this result, in addition to choosing $\calF$ from an expressive class of neural networks in a way that $\text{span}(\calF)$ is dense in the space of continuous functions, makes it possible to prove that $\W_{1,\calF}$ is actually a metric \citep{liu2017approximation,zhang2017discrimination}. In \autoref{prop:W2F-metric}, we study the moment-matching property for $\W_{2,\calF}(\mu,\nu)$. The proof appears in \autoref{apdx:prop-W2F-metric}. For $\calF \subset \cvx(\calX)$ parametrized with $\Theta\subset \mathbb{R}^M$, and for any $\theta=(\theta_1,\ldots,\theta_M) \in \Theta$, define the tangent space as 
\begin{equation}\label{eq:Tan-space}
\Tan_\theta \calF\coloneqq \text{span}\left\{\frac{\partial f}{\partial \theta_m}(\cdot;\theta):~\forall m \in[M] \right\} .
\end{equation} 

\begin{theorem}\label{prop:W2F-metric}
	Let $\Theta_0\coloneqq\{\theta \in \Theta:~f(\cdot;\theta)=\frac{1}{2}\norm{\cdot}^2\}$ and assume it is non-empty. 
	Considering~\eqref{eq:W2F}, 
	\begin{equation*}
	\W_{2,\calF}(\mu,\nu) =0 \quad \implies\quad\mu \equiv_{\Tan_{\theta_0}\calF} \nu,
	\end{equation*}
for any $\theta_0 \in \Theta_0$ that belongs to interior of $\Theta$.
\end{theorem}
The moment-matching property for $\W_{1,\calF}$ is global, in the sense that it holds for all functions $f \in \calF$, whereas the moment-matching property for $\W_{2,\calF}$ is local (restricted to tangent spaces). Therefore, the moment-matching property for $\W_{1,\calF}$ is stronger. Note that if $\calF$ is a convex cone, we have $\Tan_\theta \calF = \calF$ for all $\theta\in\Theta$, and the results for $\W_{1,\calF}$ and $\W_{2,\calF}$ are the same.

\subsection{Embedding and Restricted Approximibality}\label{sec:app-ability}
Low distortion embeddings find applications in various areas; in learning embeddings \citep{courty2017learning}, in devising a k-nearest neighbor strategy, or in forming distance matrices for further statistical analysis (e.g., clustering.) It is also important in GANs to prevent mode-collapse by guaranteeing that the learned generated distribution is close to the underlying real distribution in the exact distance. 

A question of this nature, termed as {\em restricted approximability}, has been posed and answered by \citet{bai2018approximability} for the case of $\W_{1,\calF}$. The main idea is that for any given class of functions $\calF$, there exists a class of distributions $\calD$ such that the approximate distance is accurate for any two distributions belonging to $\calD$. However, their result require assuming densities and the proposed modified approximation (through Gaussian convolutions) for dealing with general distributions is not easy to compute. 

We show the notion of approximability for $\W_{2.\calF}$. For a distribution $\mu$ and a class of functions $\calF \in \cvx(\calX)$, let $\nabla \calF \#\mu \coloneqq \{T \# \mu:~T(x)\in \partial f(x)~\ud \mu\text{-a.e.},~\forall f \in \calF\}$ be a class of distributions generated from the push-forward of $\mu$ by gradients of all functions in~$\calF$. Define the $\W_2$-projection of $\mu$ onto a subset of distributions $\calD$ as $\operatorname{Proj}(\mu;\calD)\coloneqq \argmin_{\nu \in \calD}~\W_2(\mu,\nu)$. In the following result, we prove an upper-bound on the exact metric between $\mu$ and $\nu$, in terms of the distance between $\nu$ and $\operatorname{Proj}(\nu;\nabla \calF \# \mu)$. The proof relies on~\autoref{thm:Brenier-main}-\eqref{thm:Brenier-main-2} and appears in \autoref{apdx:prop-W2F-approximability}.
 
\begin{theorem}\label{prop:W2F-approximability}
	Consider $\W_{2,\calF}(\mu,\nu)$ in~\eqref{eq:W2F} where $\mu$ and $\nu$ do not necessarily admit densities. 
	\begin{enumerate}[(i)]
	\item \label{prop:W2F-approximability-1}
	If $\nu \in \nabla \calF\#\mu $, then $\W_2(\mu,\nu)=\W_{2,\calF}(\mu,\nu)$.
	
	\item \label{prop:W2F-approximability-2}
	Assume $\norm{\nabla f^\star(x)-x}_2\leq c_1\norm{x}_2 + c_2$ for all $x \in \calX$ and for all $f\in \calF$. Then, 
	\begin{equation}\label{eq:W2F-approximability}
	\W_{2,\calF}(\mu,\nu) \leq 
	\W_2(\mu,\nu)~\leq~ \left(\W^2_{2,\calF}(\mu,\nu) + c\epsilon\right)^{1/2} + \epsilon 
	\end{equation}
	where 
	$\epsilon=\W_2(\lambda,\nu)$, 
	$\lambda =\operatorname{Proj}(\nu;\nabla \calF \# \mu)$, 
	$c=\frac{c_1}{2}(\sigma_\nu + \sigma_\lambda) +c_2$, 
	$\sigma_\nu \coloneqq (\int x^2\ud \nu)^{1/2}$, and 
	$\sigma_\lambda \coloneqq (\int x^2\ud \lambda)^{1/2}$. 
	\end{enumerate}
\end{theorem}

\begin{remark}
	Further bounding the right-hand side of~\eqref{eq:W2F-approximability} provides 
	\begin{equation*}
	\W_{2,\calF}(\mu,\nu) \leq \W_2(\mu,\nu)~\leq~ \W_{2,\calF}(\mu,\nu) + \min\{\sqrt{c\epsilon}, \frac{c\epsilon} {2\W_{2,\calF}(\mu,\nu)} \}+ \epsilon,
	\end{equation*} 
	which helps in better understanding the asymptotic behavior of the upper-bound as $\epsilon \to 0$. 
\end{remark}

The result of \autoref{prop:W2F-approximability} can be used in design and analysis of generators and discriminators in GAN. In particular, for any given discriminator class $\calF$ and a generated distribution $\mu$, one can compute the class of distributions $\nabla \calF \# \mu$ whose distance to $\mu$ can be accurately approximated. 
 
As an illustrating example, consider the problem of learning a symmetric one-dimensional bimodal delta distribution $\ud \nu = \frac{1}{2}\delta_{\{x=-v\}} + \frac{1}{2}\delta_{\{x=v\}}$. Suppose the generator generates distributions of the form $\ud \mu(x) = \frac{1}{2}\delta_{\{x=-u\}} + \frac{1}{2}\delta_{\{x=u\}}$ where $u\in \real $ is the parameter of the generator. The parameter $u$ is learned by minimizing $\W_{2,\calF}(\mu,\nu)$ where the discriminator function class $\calF \coloneqq \{f:x\mapsto \max(\sigma^2(x-w),\sigma^2(-x-w)):~|w|\leq L\}$ where $\sigma(x)$ is the ReLU function. In this case, we can show that all symmetric bimodal delta distributions $\nu \in \nabla \calF \# \mu$ for all $u,v\in \real$ such that $|u-v|\leq L$ (details appears in \autoref{sec:restr-ICNN}). As a result of \autoref{prop:W2F-approximability}-\eqref{prop:W2F-approximability-1}, $\W_{2,\calF}(\mu,\nu)=\W_2(\mu,\nu)= |u-v|$ is exact. Moreover, if $\ud \nu = (\frac{1}{2}-\alpha)\delta_{\{x=-v\}} + (\frac{1}{2}+\alpha)\delta_{\{x=v\}}$ is slightly varied and does not belong to $\nabla \calF \# \mu$, then \autoref{prop:W2F-approximability}-\eqref{prop:W2F-approximability-2} provides an upper-bound for the error, with $\epsilon \leq 2\alpha|v|$, and $c=|L|$.

\subsection{Statistical Generalization}\label{sec:muN-mu}
Here, we study the generalization properties of the approximate metric. In particular, we are interested in studying the rate of convergence of $\W_{2,\calF}(\mu^{(N)},\mu)$ to zero as $N \to \infty$ where $\mu^{(N)}\coloneqq \frac{1}{N}\sum_{i=1}^N\delta_{X^i}$ is the empirical distribution formed from independent samples $\{X^i\}_{i=1}^N$ from~$\mu$. 

The rate has been known for exact Wasserstein distances; e.g., $O(N^{-\frac{1}{d}})$ for $\W_1$ \citep{dudley1969speed} and $O(N^{-1/(d+4)})$ for $\W_2$ \cite[Section 10.2]{rachev1998mass}. It is implied form the rate that in order to achieve $\epsilon$ error, the number of required samples should increase exponentially with the dimension. In fact, \citet{arora2017generalization} showed that for a Gaussian distribution $\mu$, $\W_{1}(\mu^{(N)},\mu)\gtrsim 1$ with high probability if the number of samples grow polynomially with the dimension. In contrast to the exact Wasserstein distance, the convergence holds for the approximate $L^1$-Wasserstein distance $\W_{1,\calF}$ and approximate $f$-divergences \citep{arora2017generalization,zhang2017discrimination}.
 
Here, we are interested in studying the convergence rate for $\W_{2,\calF}$ and we follow a Rademacher complexity argument similar to \citet{zhang2017discrimination}. The proof of the following result appears in \autoref{apdx:prop-generalization}.

\begin{theorem}\label{prop:generalization}
	Consider$\W_{2,\calF}(\mu,\nu)$ defined in~\eqref{eq:W2F} where $\mu$ and $\nu$ have finite second order moments. For $X^i\sim \mu$ and $Y_j\sim \nu$, let $\mu^{(N)}\coloneqq \frac{1}{N}\sum_{i=1}^N \delta_{X^i}$ and $\nu^{(N)}\coloneqq \frac{1}{N}\sum_{i=1}^N \delta_{Y^i}$. Then,
	\begin{equation}\label{eq:gen-bnd}
	\frac{1}{2}\Expect \left[ \left| \W_{2,\calF}^2(\mu^{(N)},\nu^{(N)}) -\W_{2,\calF}^2(\mu,\nu)\right| \right]~\leq~ \mathcal{R}_N(\frac{1}{2}\norm{\cdot}^2-\calF,\mu) + \mathcal{R}_N(\frac{1}{2}\norm{\cdot}^2-\calF^\star,\nu)
	\end{equation}
	where the expectation is over all possible sample sets $X^1,\ldots,X^N$ (drawn i.i.d.~from $\mu$) and $Y^1,\ldots,Y^N$ (drawn i.i.d.~from $\nu$), and $\mathcal{R}_N(\calF,\mu)$ denotes the Rademacher complexity of the function class $\calF$ with respect to $\mu$ for sample size $N$. 
\end{theorem}

As an example, consider $\calF \coloneqq \{f:x\mapsto \frac{1}{2}\norm{x}^2 + w^\top x:~\norm{w}_2\leq L \}$. Then, computing the Rademacher complexity of the class and using the result of~\autoref{prop:generalization} yields 
$\frac{2L}{\sqrt{N}}(\sqrt{\trace{\Sigma_\mu}} + \sqrt{\trace{\Sigma_\nu}})+ O(\frac{1}{N})$ for the right-hand side of \eqref{eq:gen-bnd} where $\Sigma_\mu = \int x x^\top \ud \mu(x)$ and $\Sigma_\nu = \int x x^\top \ud \nu(x)$. On the other hand, using the analytical solution that is available for this special case, yields $|\W^2_{2,\calF}(\mu^{(N)},\nu^{(N)}) - \W^2_{2,\calF}(\mu,\nu)| \leq \frac{1}{\sqrt{N}} \norm{m_\mu-m_\nu}_2 \sqrt{\trace{\Sigma_\mu}+\trace{\Sigma_\nu}} + O(\frac{1}{N})$ where $m_\mu = \int x \ud \mu(x)$ and $m_\nu = \int x \ud \nu(x)$.

\subsection{Approximate Transport Map}\label{sec:approx-map}
The optimal transport map is approximated in the regularization-based approaches \citep{seguy2017large} by computing the Barycentric projection map from the approximate optimal coupling. However, it is difficult to show that the approximation satisfies any specific properties. In the following, we provide a characterization for the approximate transport map we define in~\eqref{eq:approx-opt-map}. We show that the push-forward of one of the marginals with such approximate map has certain moments that are equal to the moments of the other marginal. The proof follows from the first-order optimality condition in \autoref{prop:J-derivative}.
\begin{theorem}\label{prop:approximate-map}
	Suppose \autoref{assump} holds. Let $\bar{\theta} \in \argmin_{\theta\in \Theta}\tilde{J}(\theta)$ and belongs to interior of $\Theta$ where $\tilde{J}$ is defined in~\eqref{eq:min-theta}. Then,
	\begin{equation*}
	\int g(x) \ud \mu(x) = \int g(T_\calF(y; \bar\theta))\ud \nu(y)
	\end{equation*}
	for all $g \in \Tan_{\bar{\theta}}\calF$. 
\end{theorem}
As an example, consider $\calF\coloneqq \{f:x\mapsto \sigma^2(w^\top x + b); w \in \real^d, b\in \real \}$ where $\sigma(x)$ is the ReLU function. Then, the approximate transport map preserves the moments generated by $x\sigma(w^\top x + b)$ and $\sigma(w^\top x + b)$ for $(w,b)$ that achieves the minimum. 

\subsection{Further Results on Conic Restrictions}\label{sec:conic-summary}
In this section, we consider the special case where the restriction is over a class of convex functions $\calF$ that from a convex cone. In this special case, strong duality holds for the restricted optimization problem~\eqref{eq:min-theta}. As a result, it is possible to obtain stronger results about the theoretical properties of the approximation. \autoref{prop:conic-all} provides a subset of such results. A more comprehensive treatment is given in \autoref{sec:restr-orig}.

The subset of functions $\calF \subseteq \cvx(\real^d)$ is a convex cone if $\forall f, g \in \calF$ we have $\alpha f + \beta g \in \calF$ for all $\alpha,\beta \geq 0$. Define a {\em preorder} $\preceq_\calF$ (a reflexive and transitive relation) on $M_+(\real^d)$ according to
\begin{align*}
\mu \preceq_\calF {\nu} \quad \Leftrightarrow\quad \int f(x) \ud \mu(x) \leq \int f(x) \ud {\nu}(x),\quad \forall f \in \calF
\end{align*}
for any $\mu,\nu \in M_+(\real^d)$. The proof of the following result is given in \autoref{app:proof-prop:conic-all}.

\begin{theorem}\label{prop:conic-all}
Consider the approximate metric~\eqref{eq:W2F}. Assume $\calF \subset \cvx(\real^d) $ is a convex cone and $\norm{\cdot}_2^2 \in \calF$. Then,
\begin{enumerate}
		\item Duality:
\begin{equation}\label{eq:W2F-duality}
\W^2_{2,\calF}(\mu,\nu) 
\,=\, \inf_{\lambda \preceq_\calF \mu} ~\left[\W_2^2(\lambda,\nu) + \int \frac{1}{2}\norm{x}_2^2 \ud \mu(x) -\int \frac{1}{2}\norm{x}_2^2 \ud \lambda(x) \right] 
\,\geq\, \inf_{\lambda \preceq_\calF \mu}~\W_2^2(\lambda,\nu) .
\end{equation}
	\item Moment matching:
\begin{equation*}
\W_{2,\calF}(\mu,\nu) =0 \quad \iff \quad \mu \succeq_\calF \nu\quad \text{and}\quad \int \norm{x}_2^2\ud \mu(x) = \int \norm{x}_2^2\ud \nu(x).
\end{equation*}
Note that $\W_{2,\calF}$ is not necessarily symmetric with respect to its two arguments. 

\item Embedding (Approximability): If $\nu \in \nabla \calF \# \mu$, then $\W_{2,\calF}(\mu,\nu) = \W_2(\mu,\nu)$. Otherwise, 
\begin{equation*}
\W_{2,\calF}(\mu,\nu)\leq \W_{2}(\mu,\nu) \leq \left(2\W^2_{2,\calF}(\mu,\nu) + 2\epsilon^2 \right)^{1/2} +\epsilon ,
\end{equation*}
where $\epsilon \coloneqq \inf_{\lambda \in \nabla \calF\# \mu} \W_2(\lambda,\nu)$. 
\end{enumerate}
\end{theorem}
Note that \eqref{eq:W2F-duality} can be alternatively expressed as 
\[
\W^2_{2,\calF}(\mu,\nu) =\int \frac{1}{2}\norm{x}_2^2 \ud \mu(x) -
\sup_{\lambda \preceq_\calF \mu} ~\left[\int \frac{1}{2}\norm{x}_2^2 \ud \lambda(x) - \W_2^2(\lambda,\nu) \right]. 
\]

One of the insightful examples of a parameterized class of functions that also form a convex cone is the class of convex quadratic functions. We discuss this class in detail in~\autoref{sec:restr-quad}. The class of quadratic functions is also studied in the context of GAN by \citet{feizi2017understanding}. .

\section{Practical Implications}\label{sec:practical}
In this section, we list a few practical implications of our proposal, namely restricting the dual Kantorovich form to parametrized sets of convex functions for the purpose of approximating $\W_2$ and the optimal transport map.

\subsection{Flexibility in Approximation}\label{sec:flexible}
The proposed approximation strategy provides a great deal of control over the statistical and computational properties of the approximations. Informed by the effects of these choices, characterized in \autoref{sec:main-results}, one can adapt the restriction set to the requirements of the underlying problem in which one wishes to use the approximate metric or the approximate transport map; e.g., \citep{rabin2011removing,rabin2011wasserstein}. This is in contrast with the regularization-based methods in which only special regularization functions can be used (those for which we have fast algorithms, hence by now mostly limited to entropic regularization and $\ell_2$ regularization) and there only is a single knob, namely the regularization parameter $\lambda$ in \eqref{eq:reg-continuous}, that controls the bias, the accuracy, etc. 

\subsection{Faster Optimization via Homotopy; A Progressive Training for GANs}\label{sec:homotopy}
The idea of warm-starting a procedure is prevalent in machine learning; from alleviating the cold-start problem in recommendation systems to regularized loss minimization. For example, in the latter, the idea is to start from a large regularization parameter~$\lambda$ and progressively decrease~$\lambda$. Then, exact homotopy path-following methods \citep{osborne2000lasso,osborne2000new,efron2004least} and approximate homotopy continuation methods \citep{hale2008fixed,xiao2013proximal} provide low iteration complexity as well as low per-iteration cost in convex optimization. As discussed in \autoref{sec:flexible}, we have in the proposed framework a (more flexible) way for controlling the complexity of the solution by changing the restriction set (compared to varying $\lambda$ above.) Therefore, in approximating $\W_2$ and the transport map, we can begin with a simple parametrized set (say a simple ICNN) and use the optimal parametrized function in each stage for warm-starting the optimization process \eqref{eq:W2F} (e.g., training a slightly larger ICNN) in the next stage. In the context of GANs and evaluating the distance within the discriminator, as training goes forward, we can make the discriminator family more complicated and keep the moment-matching property (see \autoref{sec:moment-matching}) along the way. 

\subsection{Enhancing the Generator using the Discriminator; Compositional GANs}\label{sec:composite-wgan}
Brenier theorem provides the optimal transport map as a byproduct of computing the $\W_2$ distance, and can be used in learning generative models as discussed next. Note that computing the optimal transport map is not straightforward when other divergence functions are used which makes this proposal very suitable to the case of $L^2$-Wasserstein distance examined in this paper. 

\newcommand{\new}[1]{\color{blue}#1\color{black}}
Had we were able to solve the Monge's optimal transport problem (given in \eqref{eq:Monge-form}) we could have generated real-looking samples by applying the optimal Monge map (a {\em deterministic} function $T_M$ corresponding to a Kantorovich plan $\ud \pi(x,y) = \ud \mu(x) \delta (y=T_M(x))$; as in contrast with a stochastic coupling that may split mass) to the low-dimensional Gaussian samples. This is similar to the approach of \cite{mesa2018distributed} which is computationally- and memory-expensive especially when used with real data such as in large-scale image classification tasks. With GANs, we alternatively learn a generator function $\mathcal{G}$ that transforms a low-dimensional Gaussian distribution $\gamma$ into a distribution $\mathcal{G}\#\gamma$ that is as close (in a sense specified by a divergence) to the high-dimensional data distribution $\rho$ as possible. 
 
Now, with a GAN-based approach, suppose that we have found a {\em deterministic} optimal transport map $T$ when transporting the outputs of the generator (samples from $\mathcal{G}\#\gamma$) to real samples (from $\rho$) by solving the Kantorovich problem \eqref{eq:dual-form-main}. Then, {\em the composition of this map (which is implementable as a function) with the generator} can be applied to the Gaussian samples, namely samples from $(T \circ \mathcal{G})\#\gamma$, in order to generate images that are as close as possible in distribution to real images (they may not coincide as the generator may not be expressive enough.) This allows for {\em enhancing the learned generator} through no additional efforts in design. If the marginal distributions admit a density or if the optimal $f^\star$ in the dual problem is differentiable, then we get a deterministic map $\nabla f^\star$ from \eqref{eq:approx-opt-map}; see \autoref{thm:Brenier-main}-\eqref{thm:Brenier-main-3} for the former and \autoref{thm:Brenier-main}-\eqref{thm:Brenier-main-2} for the latter. However, even if the map is not deterministic, one can use the optimization problem in \eqref{eq:nablafs} to compute a subgradient which can then be used in the composition; see \autoref{thm:Brenier-main}-\eqref{thm:Brenier-main-2}.  
Here, considering the approximation $\W_{2,\calF}$ allows for guaranteeing a deterministic map, as discussed next, while also being computationally efficient (depending on $\calF$ and $\Theta$). 

Consider a parametrized family $\calF$ of {\em strictly} convex functions (e.g.  $\frac{\eta}{2}\norm{\cdot}_2^2$ added to a family of ICNN) and use the corresponding approximate distance $\W_{2,\calF}$ for the discriminator in GAN. It is well-known that the convex conjugate to a strictly convex function is differentiable. Hence, {\em we get a deterministic transport map as in \eqref{eq:approx-opt-map} by design}. 
Moreover, the distance and the map are now computable, as opposed to the true distance and map, thanks to the approximation machinery. 
The only remaining tradeoff is the speed of convergence in computing $\nabla f^\star$ (solving~\eqref{eq:nablafs}), which depends on how strictly convex the functions in $\calF$ are, and the accuracy of approximations.

In summary, we propose a modular and interpretable understanding for $L^2$-Wasserstein GANs which connects the generator training with discriminator computations (the distance) to allow for learning an overall composite generator. One of the two parts, the one inside the discriminator, represents a convex function for which we know of an extensive analysis. The other part (the generator), thanks to the eventual enhancement via composition, can now be assigned less complexity, allowing for faster training and better interpretability. In fact, with this approach, there is a way for generator and the discriminator to tradeoff each other’s complexity. This tradeoff can also be seen as a more accurate game description for GANs compared to the ``generation and 0/1-discrimination'' picture. 

\paragraph{Post-processing a GAN.}
The above procedure can also be used after a generator has been fully trained: 1) compute $\W_{2,\calF}$ (with $\calF$ prescribed above) between the output of the generator and real samples; a distance computation in a stochastic optimization manner. 2) compose the original generator with the approximate transport map from $\W_{2,\calF}$.

\section{Prior Art}\label{sec:prior}
Here, we provide a brief comparison with existing methods in (mostly) continuous computational optimal transport; through regularization or other approximation techniques. 

\paragraph{Regularization-based Approaches.} 
A family of algorithms consider the primal optimization problem in~\eqref{eq:dual-form-main} with entropic regularization, namely 
\begin{equation*}
\inf_{\pi \in \Pi(\mu,\nu)}~ \int c(x,y)\ud \pi(x,y) + \lambda \int \log(\pi(x,y)) \ud \pi(x,y),
\end{equation*} 
where $\lambda >0$ is the regularization parameter. The dual form of such a problem is given by 
\begin{align}\label{eq:reg-continuous}
\sup_{f,g}~\int f(x) \ud \mu(x) + \int g(y)\ud \nu(y) - \lambda \int \exp\Bigl(\frac{f(x)+g(y) - c(x,y)}{\lambda} \Bigr) \ud\mu(x) \ud \nu(y) .
\end{align}
Compared to the dual form~\eqref{eq:dual-form-main}, the constraint set is removed and a penalty term is added in its place. In the discrete setting, the problem can be solved using the Sinkhorn iteration algorithm or other methods \citep{cuturi2013sinkhorn,dvurechensky2018computational}. In the continuous setting, the optimization may be restricted to a parametrized class of functions, e.g., the RKHS class \citep{genevay2016stochastic} or neural networks \citep{seguy2017large}, and solved using stochastic optimization algorithms. The optimal solutions to~\eqref{eq:reg-continuous} can then be used to get an optimal coupling~$\bar{\pi}$ (e.g., see \citet[Proposition 2.1]{genevay2016stochastic}) which is then used to solve $\min_T\int c(y,T(x))\ud \bar{\pi}(x,y)$ to get a Barycenter projection map. This problem is also solved using a stochastic optimization algorithm where the map $T$ is parametrized as a deep neural network \citep{seguy2017large}. 

The regularization introduces a bias error in estimation. This leads to inexact estimates of the metric and noisy maps that, for example, lead to blurry images for applications in image processing \citep{essid2018quadratically,blondel2017smooth}. Moreover, decreasing the regularization parameter to decrease the bias results in slow convergence and numerical instability; see \citet{schmitzer2016stabilized,dvurechensky2018computational} and \citet[Remark~4.6]{peyre2017computational}. One can replace the entropic regularization with a strictly convex penalty term; e.g., a quadratic \citep{essid2018quadratically}. This has the advantage of producing sparse couplings instead of dense couplings we expect from entropic regularization \citep{blondel2017smooth}. However, the projection step in the optimization algorithm becomes computationally expensive, and this leads to less efficient algorithms compared to the Sinkhorn algorithm \cite[Remark 4.8]{peyre2017computational}.

Furthermore, the barycenter projection map parametrized with a deep neural network as in \citet{seguy2017large}) is inherently continuous, while the exact transport map  for the real data that usually has a complicated support (a non-convex union of low-dimensional manifolds \citep{arjovsky2017wasserstein,guo2019mode}) is not continuous. While a discontinuous map may be approximated by a {\em big enough} network, the aforementioned insight calls for a better modeling approach. In fact, using $\W_2$ and the Brenier theorem proposed in this paper, allows for learning convex (continuous) potentials whose gradient mapping are now to represent the transport map and can be discontinuous.

\paragraph{Approximating the $L^1$-Wasserstein Distance.}
An approximation to $\W_1$ can be defined as
\begin{equation}\label{eq:L1-restricted}
\W_{1,\calF}(\mu,\nu) = \sup_{f \in \calF}~~ \int f(x) \ud \mu(x) - \int f(y) \ud \nu(y)
\end{equation} 
where $\calF$ is a subset of Lipschitz functions from $\calX$ to $\mathbb{R}$. The approximation is exact if $\calF$ contains all $1$-Lipschitz functions; see~Equation (7.1) in \citet{villani2003topics}. In WGAN \citep{arjovsky2017wasserstein}, $\calF$ is chosen to be the set of functions parameterized as neural networks. Since projecting a network onto the set of $1$-Lipschitz functions is not straightforward, various techniques such as constraining the weights to bounded sets have been used \citep{gulrajani2017improved, salimans2018improving,wei2018improving}; but could lead to unused capacity and exploding or vanishing gradients \cite[Section 3]{gulrajani2017improved}. In contrast, in the $L^2$ setting, we work with the set of convex functions, and expressive representations such as ICNNs are easy to project to, namely by thresholding weights by zero. 

Another challenge in solving~\eqref{eq:L1-restricted}, to approximate $\W_1$, 
is that the optimal weights are usually achieved at the boundary of the optimization domain which makes many optimization algorithms slower to converge. Let us make this notion more rigorous via an example. 

\begin{example}\label{ex:interior-L1-L2}
Consider~\eqref{eq:L1-restricted} with $\calF_1\coloneqq \{f:x\mapsto w^\top x;~w\in\mathcal{A}\}$ where $\mathcal{A}$ is compact. Then, 
\[\W_{\calF_1,1}(\mu,\nu) = \sigma_{\mathcal{A}} (m_\mu-m_\nu) \quad\text{and}\quad
w_{\op{opt}}= \argmax_{z\in \mathcal{A}} \,\langle x,z\rangle,\] 
where $m_\mu = \int x \ud \mu(x)$, $m_\nu = \int x \ud \nu(x)$, 
and $\sigma_{\mathcal{A}}(x) \coloneqq \sup_{z\in \mathcal{A}} \,\langle x,z\rangle$ is the support function for $\mathcal{A}$. 
Observe that the optimal weight vector $w_{\op{opt}}$ belongs to the boundary of $\op{conv}(\mathcal{A})$. In contrast, consider~\eqref{eq:min-theta} with $\calF_2\coloneqq\{f:x\mapsto \frac{1}{2}\norm{x}^2+ w^\top x;~w \in \mathcal{A}\}$. Then, 
\[w_{\op{opt}}= \op{Proj}( m; \mathcal{A}) 
\quad\text{and}\quad\W_{\calF_2,2}^2(\mu,\nu) = \frac{1}{2}\|m\|_2^2 - \frac{1}{2}\|m - w_{\op{opt}} \|_2^2,\] for $m = m_\nu - m_\mu$ and where 
$\op{Proj}(x;\mathcal{A}) \coloneqq \argmin_{z\in\mathcal{A}}\|x-z\|_2$ is the orthogonal projection onto~$\mathcal{A}$. Observe that $w_{\op{opt}}$ is not necessarily at the boundary of $\mathcal{A}$; e.g., if $\mathcal{A}$ contains $m$ in its interior. 
\end{example}
\autoref{ex:interior-L1-L2} indicates less sensitivity of the latter method to the choice of $\calF$. Last but not least, in GANs, the $L^2$-Wasserstein distance (or an approximation) has been shown to be beneficial in the study of the dynamics of the generator and obtaining natural gradient flows \citep{lin2019wasserstein,jacob*2019wgan}.

\paragraph{Other Approximation Techniques.} 
Aside from regularization-based approached discussed above, a variety of other approximation methods have been proposed in the literature, including but not limited to: Sliced Wasserstein distance computed from random one-dimensional projections of the data \citep{rabin2011wasserstein,bonneel2015sliced}, fluid-dynamics based approaches \citep{benamou2000computational}, multi-level grid methods \citep{liu2018multilevel}, and embedding methods where the samples are embedded in lower dimensional spaces \citep{courty2017learning}. Computing a transport map (not necessarily optimal) in continuous settings appears in \citet{el2012bayesian,heng2015gibbs,mesa2018distributed}. The approximation of Earth Mover's Distance has also been considered in the literature; \citep{indyk2003fast,shirdhonkar2008approximate}.
%
Approximation of the $L^2$-Wasserstein distance using convex geometric tools appears in \citet{lei2018geometric}. 

Modified constraint sets in the optimal transportation problem have also been studied before. 
In \citet{korman2015optimal} the set of couplings $\Pi(\mu,\nu)$ is constrained to be all joint distributions with marginals $\mu$ and $\nu$ that are dominated with a predefined measure (i.e., a capacity constraint), hence the constraint set becomes smaller. Whereas, the set of couplings studied here (e.g., $\Pi^\calK_\equiv(\mu,\nu)$) are larger than the original $\Pi(\mu,\nu)$ (see \autoref{lem:Pi-union}). The idea of enlarging the feasible space for the primal transport problem, in~\eqref{eq:Kantorovich-form}, has appeared before in other forms; e.g., see \cite{beiglbock2009general}. By restricting the function classes to some $\calF$ and $\calG$, we grow the set of joint distributions to those consistent with $\mu$ and $\nu$ in the more general sense defined in this work. \cite[Section 4.6]{rachev1998mass} discusses the primal optimal transportation problem where the joint distribution is constrained to have certain moments in addition to satisfying marginal constraints. See also \citet{zaev2015monge}.

\citet{guo2019mode} propose to approximate the Brenier potential with piecewise affine functions directly from the given samples; through Alexandrov’s solution to the Minkowski problem. However, such construction, while elegant, seems to be computationally expensive. More specifically, their algorithm based on the solution of \citet{gu2016variational} to the Minkowski problem constructs a piecewise affine Brenier potential with $N$ pieces ($N$ being the number of samples, which can be very big) through a second-order optimization approach (Newton’s method) in which the computation of gradient and the Hessian may require maintaining a triangulation.

\section{Numerical Optimization}\label{sec:numerics}
In evaluating the approximation $\W_{2,\calF}$ or using it within an optimization program, or in computing the approximate transport map, we need to solve the optimization problem~\eqref{eq:min-theta}, namely 
\begin{align*}
\min_{\theta\in\Theta} ~ \tilde J_{\mu,\nu}(\theta)
\end{align*}
which is a finite-dimensional constrained non-convex non-smooth optimization problem. In the above, $\tilde J_{\mu,\nu}(\theta)\coloneqq \int f(x;\theta)\ud \mu(x) + \int f^\star(y;\theta)\ud \nu(y)$ and $\Theta$ may be a non-convex set. Even with all these difficulties, we can still use stochastic first-order methods to find a solution. For such approach to work, we need to compute unbiased estimates of the gradient for the objective. This is given in \autoref{prop:J-derivative}. The analysis is similar to \citet{chartrand2009gradient}, but the derivative is computed with respect to the function, not the parameter. The proof appears in \autoref{apdx:J-derivative}. 
\begin{theorem}\label{prop:J-derivative}
	Consider the objective function in~\eqref{eq:min-theta} and $\Theta\subset \mathbb{R}^M$. Suppose $f(x;\theta)$ is convex in~$x$ and satisfies \autoref{assump}. Then, $\nabla_{\theta_m} \tilde{J}_{\mu,\nu}(\theta) = \int_\calX \nabla_{\theta_m} f(x;\theta) \ud \mu(x) - \int_\calY \nabla_{\theta_m} f(\nabla_y f^\star(y;\theta)) \ud \nu(y)$ for all $m\in[M]$.
\end{theorem}

Using the above, we propose a numerical algorithm consisting of a nested loop:
\begin{itemize} 
\item 
	{\em An Outer loop,} a stochastic optimization algorithm, to iteratively update the parameter $\theta$ using an unbiased estimate of the derivative given by $ \nabla_\theta f(X_i;\theta) - \nabla_\theta f(\nabla_y f^\star(Y_i;\theta);\theta)$ where $\{X_i\}$ and $\{Y_i\}$ are independent samples from $\mu$ and $\nu$, respectively. It is then projected onto $\Theta$ to maintain feasibility. In practice, we use a batch of samples to sample the gradient. The stochastic nature of this strategy makes it suitable for large-scale settings. 
\item 
	{\em An Inner loop,} to compute the derivative of the convex conjugate $\nabla f^\star(y;\theta) $ via solving the convex program 
	\begin{align} \label{eq:nablafs}
	\nabla f^\star(y;\theta) = \argmax_x ~ \langle y ,x\rangle - f(x;\theta) \end{align} 
	given a value of $\theta\in \Theta$. Standard first- or second-order convex optimization algorithms may be used to solve this problem. In cases where $f(\cdot;\theta)$ admits a variational form (as in PL, PLQ, VGF \cite{jalali2017variational}, \cite{aravkin2013sparse}), saddle point optimization algorithms such as Mirror-prox can provide efficient strategies. 
\end{itemize}

\begin{algorithm}
	\KwData{ $\{X_i\}_{i=1}^N,\{Y_i\}_{i=1}^N$, a schedule of step sizes $\{\eta_k\}_{k=1}^K$, a schedule of batch sizes $\{M_k\}_{k=1}^K$}
	\KwData{ An exact oracle to compute a $g\in \partial f^\star(\cdot)$}
	Initialize $\theta_0$ randomly. 
	\For{$k=1,\ldots, K$}{
		Choose a batch $\{X_i\}_{i=1}^{M_k},\{Y_i\}_{i=1}^{M_k}$ randomly\\
		Compute $\hat{X}_i \in \partial f^\star(Y_i)$ for $i=1,\ldots,M_k$, using the given oracle\\
		Compute $u = \sum_{i=1}^{M_k} \frac{\partial f}{\partial \theta} (X_i, \theta_k) - \frac{\partial f}{\partial \theta} (\hat{X}_i, \theta_k)$ \\
		Update $\theta_{k+1} = \operatorname{Proj}(\theta_k - \eta_k u; \Theta)$ 
	}
	Return $\theta=\theta_K$.
	\caption{Projected SGD with an exact conjugate oracle; the outer loop.}
\end{algorithm}

In practice, we do not compute the derivative of the conjugate function exactly at each step of the outer algorithm. In fact, for each step of the outer loop, we run the inner loop only for a fixed number of steps, starting from the previous point from the previous step. Such a warm-start strategy reduces the computational cost of the algorithm. However, the errors introduced by this approximation are potentially structured and may harm the convergence of a plain SGD for the outer loop in more complicated cases than those with which we experimented. This motivates the use of more complicated variants of SGD and developing further understanding of the effect of such structured bias in the gradients on SGD, which we postpone to future work. 

Finally, the above optimization strategy (for evaluating $\W_{2,\calF}$ given samples) is provided to illustrate the main modules. However, the same modules can be used whenever $\W_{2,\calF}$ appears within an optimization problem. For example, when $\W_{2,\calF}$ is used as a regularization term, the optimization problem \eqref{eq:W2F} can be plugged in, to result in a saddle point optimization.

\subsection{A Numerical Example and Comparison with a Regularization-based Approach}
We provide a comparison between the proposed algorithm with the regularization-based approach proposed in \citet{seguy2017large}. We consider learning the optimal transport map between two mixtures of Gaussians and the results are depicted in \autoref{fig:outer-loop-with-reg}. We use the code provided by the authors, and both algorithms were run for the same number of epochs and per-epoch runtime is reported. Moreover, both algorithms use a 3-layer network of size $(64,128,64)$ with ReLU activations except that our ICNN has ReLU-squared in its first layer. 

\autoref{fig1} depicts the true transport map as well as the transport map learned through regularization with regularization parameters $1.0$ and $0.1$. It is observed that as the regularization parameter becomes smaller, the learned transport map gets closer to the true map. \autoref{fig2} depicts the $\ell_2$ error between the learned and the true maps as a function of the number of samples ($N$). It can be observed that the error from our method converges to zero as $O(N^{-1})$ while the error of the regularized approach (with a fixed regularization parameter) is dominated by the inherent bias due to regularization. \autoref{fig3} depicts the run-time with respect to~$N$ where the runtime of the proposed algorithm scales as $O(N)$ for each epoch, while for the regularized approach it scales as $O(N^2)$, assuming a constant batch size for both. This is due to the fact that the regularization penalty term is not separable in the two marginals, so that at each iteration, the number of required samples scales as $O(N^2)$. Finally, \autoref{fig4} plots the map estimation error against the running time. It can be observed that the proposed method lies to the left and to the bottom of the curve for the regularization-based method, hence improving both the runtime and the map estimation accuracy. 

\begin{figure}
	\centering
	\newcommand{\scl}{.45}
	\begin{tabular}{rr}
		\subfigure[]{\includegraphics[width = \scl\columnwidth]{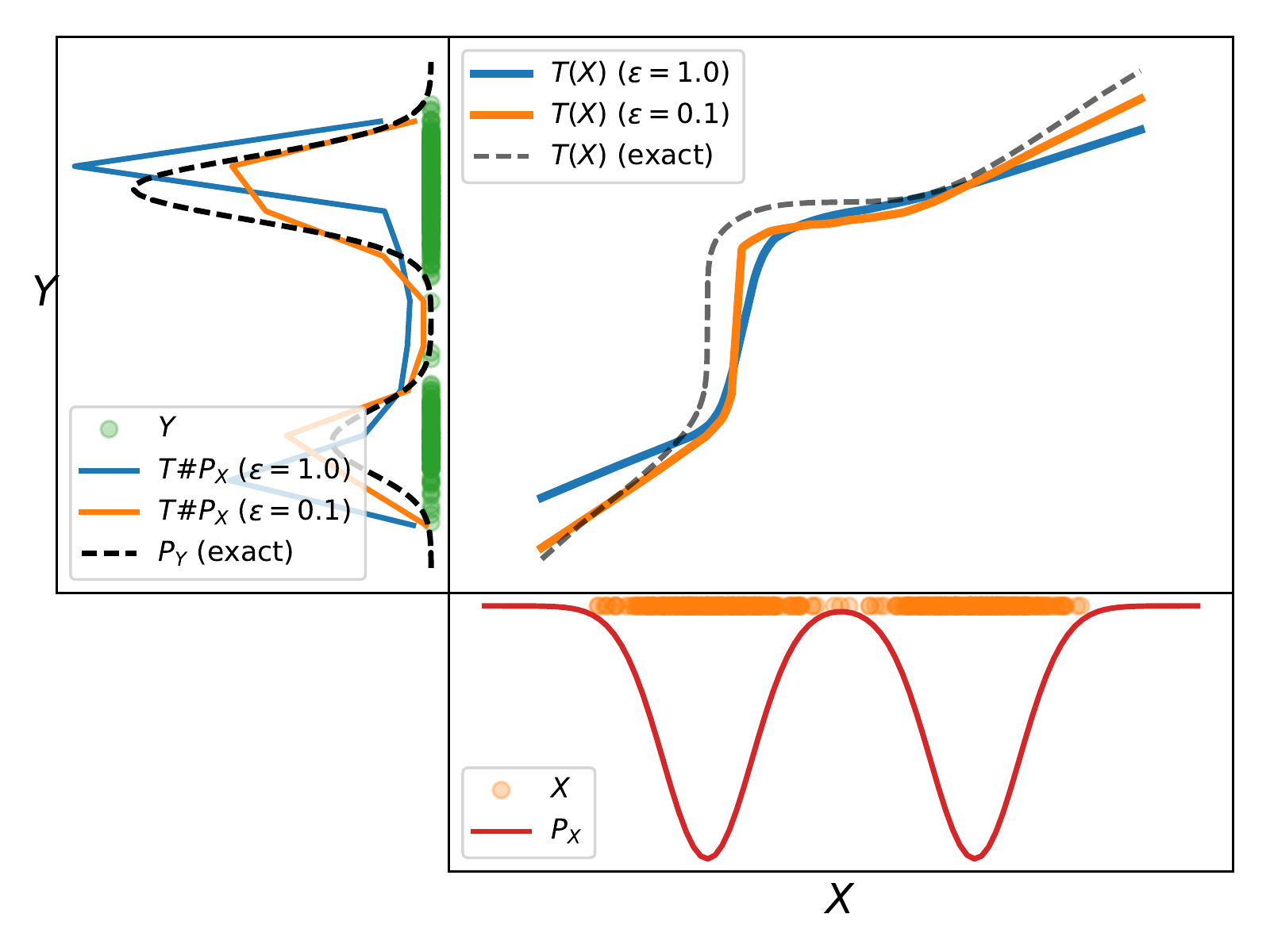}\label{fig1}} &
		\subfigure[]{\includegraphics[width = \scl\columnwidth]{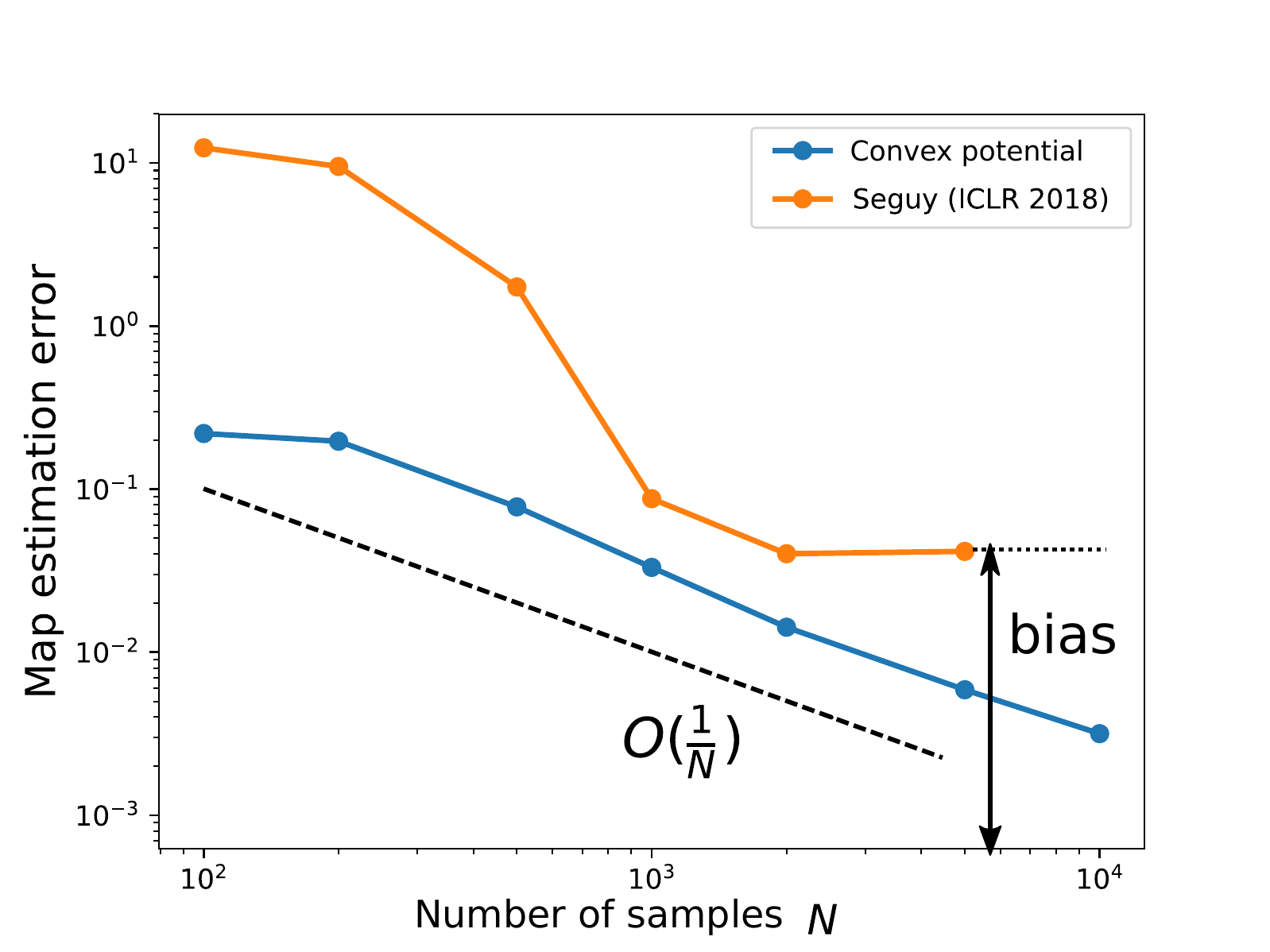}\label{fig2}} \\
		\subfigure[]{\includegraphics[width = \scl\columnwidth]{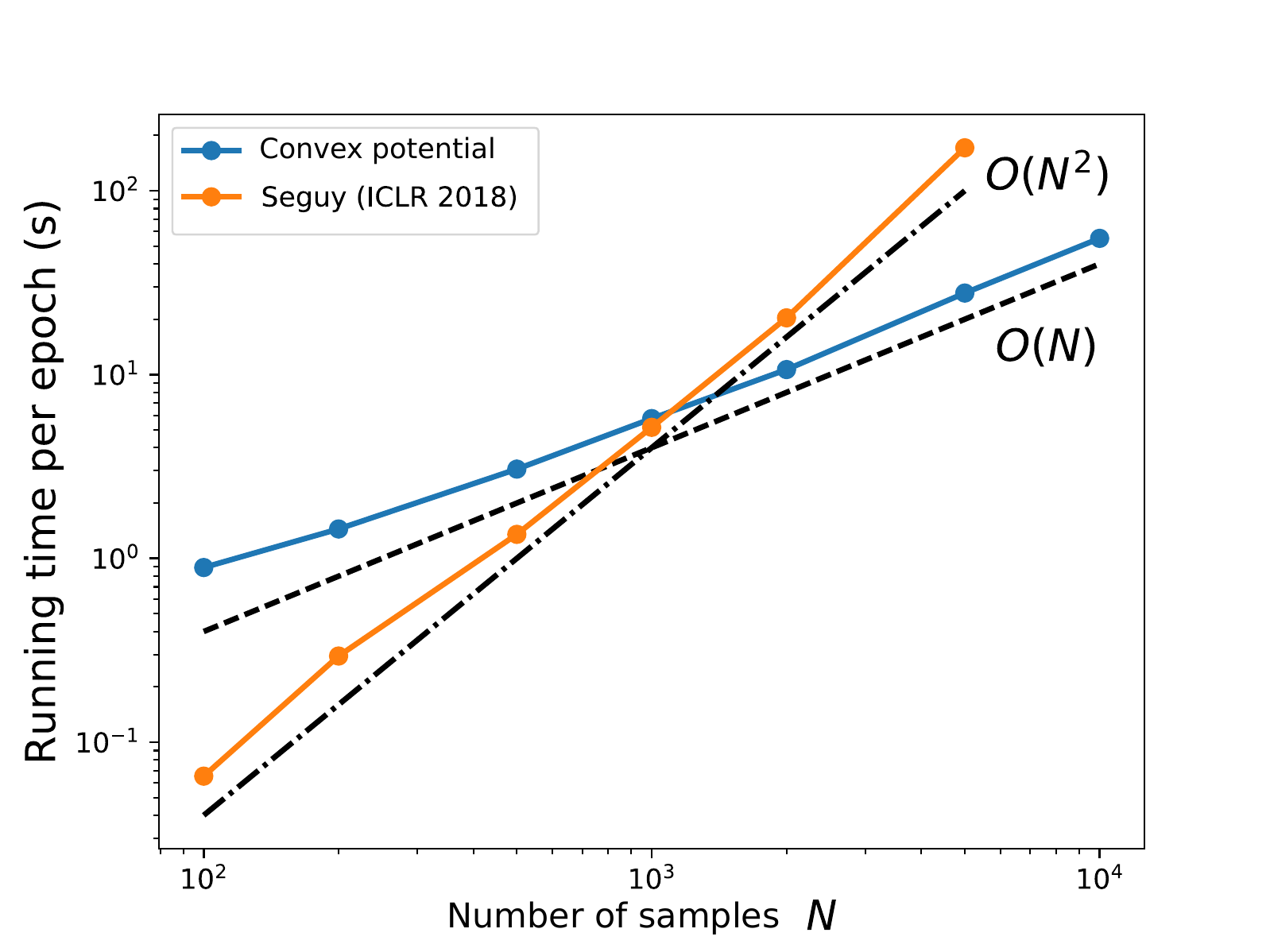}\label{fig3}} &
		\subfigure[]{\includegraphics[width = \scl\columnwidth]{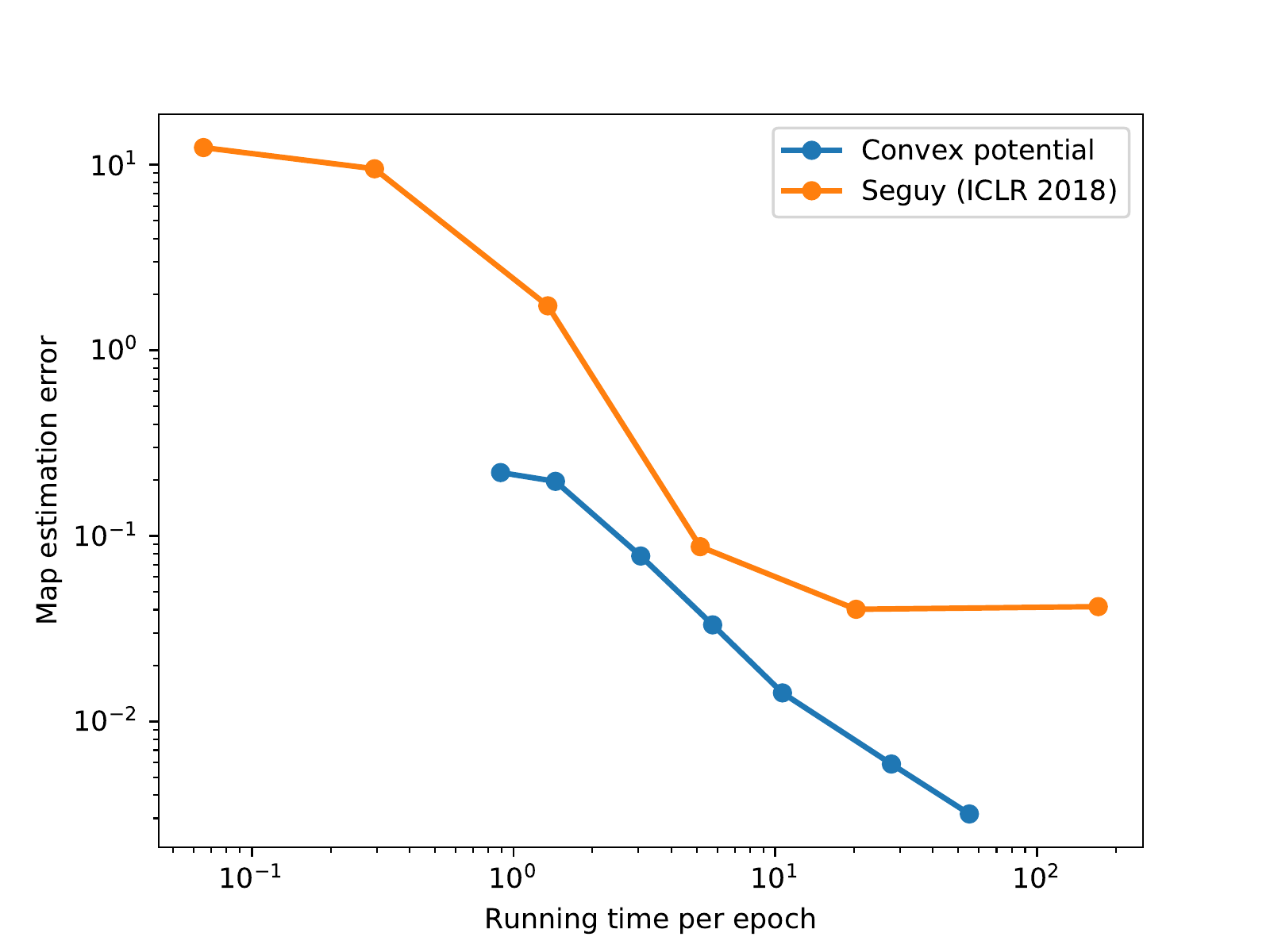}\label{fig4}}		
	\end{tabular}
	\caption{Comparison of the proposed approximation methodology with the regularization-based approach of \citet{seguy2017large}, using similar networks and the same number of epochs. 
	}
	\label{fig:outer-loop-with-reg}
\end{figure}

\newpage 
\addcontentsline{toc}{section}{References}
\bibliographystyle{abbrvnat}
\bibliography{OT-arxiv-refs}

\appendix
\newpage
\section{Background on Optimal Transport Theory}\label{app:background-OT}
This is a more detailed version of the summary provided in \autoref{sec:background-OT}. See \citet{villani2003topics} for a comprehensive overview. 
\subsection{Notation} 

\paragraph{Spaces:} With $\calX$ or $\calY$ we may denote a Polish space (a separable completely metrizable topological space) which maybe compact or not depending on the context. 

\paragraph{Measures:}
The space of Borel probability measures on $\calX$ is denoted by $\calP(\calX)$, the space of finite Borel measures by $M_+(\calX)$, and the space of signed finite Borel measures by $M(\calX)$. 
The set of probability distributions on $\calX$ with finite $p$-th order moments is denoted by $\mathcal{P}_p(\calX)$. 
The set of probability distribution that are absolutely continuous with respect to Lebesgue measure on $\real^n$, and have finite $p$-th order moments is denoted by $\mathcal{P}_{p,\op{ac}}(\real^n)$. 
The set of $d\times d$ positive definite matrices is denoted by $\mathbb{S}^n_{++}$. 
The set of probability distributions that have positive definite covariance matrices is denoted by $\mathcal{P}_{2,+}(\real^n)$.
We work with measures which are not necessarily probability distributions. Therefore, we use the integral notation instead of expectations. 

\paragraph{Functions:}
$C(\calX)$ is the space of continuous functions on $\calX$. $C_b(\calX)$ is the space of bounded continuous functions on $\calX$. They are equipped with the norm $\|\cdot\|_\infty$ where $\|f\|_\infty \coloneqq \sup_{x\in\calX} |f(x)|$ for any $f\in C_b(\calX)$. The value of the gradient of $f$ at point $x$ will be denoted by $\nabla_x f(x)$. 
The set of square integrable functions with respect to a measure $\mu$ is denoted by $L^2(\mu)$. 
The set of convex functions in $C_b(\calX)$ is denoted by~$\cvx(\calX)$. 
For a given function $f$ its convex conjugate is given by $f^\star(y)=\sup_{x} \left[\langle x, y\rangle - f(x)\right]$.

The inner product, on the space that will be clear from the context, is denoted by $\langle \cdot, \cdot\rangle$. 
For a given integer $n \geq 1$, we denote by $[n]$ the set $\{1,\ldots, n\}$.

\subsection{Optimal Transport Problem}
Let $X$ and $Y$ be two random variables on Polish spaces $\mathcal{X}$ and $\mathcal{Y}$ with (Borel) probability measures $\mu$ and $\nu$ respectively. 
The push-forward of a measure $\mu$ under a measurable map $T:\calX \to \calY$ is a measure on $\calY$, denoted by $T\#\mu$, defined according to
\begin{equation*}
(T\#\mu) (A) \coloneqq \mu(T^{-1}(A)),\quad \forall A \in \mathcal{B}(\calY)
\end{equation*}
where $\mathcal{B}(\calY)$ is the $\sigma$-algebra of Borel sets of $\calY$.
The map $T:\calX \to \calY$ is a {\it transport map} from $\mu$ to $\nu$ if $T\#\mu = \nu$. 
In the probabilistic language, $T$ is a transport map if $T(X)$ is equal to $Y$ in distribution. 
Let $\mathcal{T}(\mu,\nu)$ denote the set of all transport maps from $\mu$ to $\nu$. 
In general, there may be infinitely many transport maps between two distributions. The problem of the optimal transportation is to find a transport map that is optimal with respect to a certain cost function. Let $c:\calX \times \calY \to \real^+$ be the cost function. Then, {\em Monge's optimal transport problem} is stated as
\begin{align}
\inf_{T \in \mathcal{T}(\mu,\nu)}~ \int c(x,T(x)) \ud \mu(x)
\label{eq:Monge-form}
\end{align}
and the map that minimizes the optimization problem (if it exists) is called the {\it optimal transport map}. 

The optimal transportation problem is nonlinear and difficult to analyze. Kantorovich introduced a relaxation of the problem by minimizing over couplings of $X$ and $Y$ instead of transport maps from one to the other. A coupling of $X$ and $Y$ is a joint probability distribution $\pi$ on $\calX \times \calY$ such that its marginals are equal to $\mu$ and $\nu$, i.e., 
\begin{equation*}
\pi(A,\calY) = \mu(A),\quad \pi(\calX,B) = \nu(B),\quad \forall A\in\mathcal{B}(\calX),~\forall B \in \mathcal{B}(\calY).
\end{equation*}
The set of all couplings between $X$ and $Y$ is denoted by $\Pi(\mu,\nu)$. 
Then, 
{\em Kantorovich's optimal transport problem} is stated as
\begin{align}
\inf_{\pi \in \Pi(\mu,\nu)}~
\underbrace{\int c(x,y)\ud \pi(x,y)}_{I(\pi)}.
\label{eq:Kantorovich-form}
\end{align}

\subsection{Kantorovich Duality}
The optimization problem~\eqref{eq:Kantorovich-form} is a convex problem, i.e., both the objective and the constraint set are convex, and admits a dual formulation, namely {\em the Kantorovich's dual formulation,} given as 
\begin{equation}
\sup_{(f,g)\in\calC(c)}~\underbrace{\int f(x)\ud \mu(x) + \int g(y)\ud \nu(y)}_{J(f,g)}
\label{eq:dual-form}
\end{equation}
in which the functions $f\in L^1(\mu) $ and $g \in L^1(\nu)$ are the dual variables and $\calC(c)$ denotes the set of all measurable functions $(f,g) \in L^1(\mu) \times L^1(\nu)$ that satisfy the constraint $f(x) + g(y) \leq c(x,y)$ for $\mu$-almost all $x \in \calX$ and $\nu$-almost all $y\in \calY$; i.e.,
\begin{align}\label{eq:def-Phic-set}
\calC(c) \coloneqq \bigl\{(f,g)\in L^1(\mu)\times L^1(\nu):~f(x)+g(y)\leq c(x,y)~\ud \mu \otimes \ud \nu~\text{a.e.}\bigr\}.
\end{align}

\begin{theorem}\cite[Theorem 1.3]{villani2003topics}\label{thm:duality}
	Consider the Kantorovich's optimal transportation problem in~\eqref{eq:Kantorovich-form} and its dual formulation in~\eqref{eq:dual-form}. Assume the cost function $c$ is lower semi-continuous. Then, 
	\begin{equation*}
	\inf_{\pi\in\Pi(\mu,\nu)}\, I(\pi) ~= \sup_{(f,g)\in\calC(c)}\, J(f,g)
	\end{equation*}
	and the infimum on the left-hand side is attained.
\end{theorem}

\subsection{Wasserstein Distance}
The value of the optimization problem~\eqref{eq:Kantorovich-form} serves as distance between the two probability distributions $\mu$ and $\nu$. If the cost function is chosen to be $c(x,y) = d(x,y)^p$ where $d:\calX \times \calX\to \real^+$ is a metric on $\calX$\footnote{This can be any metric that makes the Polish space a metric space with the same topology. In general, the Polish space is not equipped with a unique metric.} and $p\in[1,\infty)$, then the resulting optimal value of~\eqref{eq:Kantorovich-form} is the Wasserstein distance of order $p$ between $\mu$ and $\nu$, namely
\begin{align*}
\W_p(\mu,\nu)\coloneqq 
\inf_{\pi \in \Pi(\mu,\nu)} ~\left[\int d(x,y)^p\ud \pi(x,y)\right]^{\frac{1}{p}}. 
\end{align*}
It is well-known that $\W_p$ is a metric on $\mathcal{P}_p(\calX)$; e.g., see  \cite[Theorem 7.3]{villani2003topics}.

\paragraph{Distance Cost Function, $p=1$.}
Consider the special case where $p=1$. Then, due to a famous result known as the Kantorovich-Rubinstein theorem, \cite[Theorem 1.14]{villani2003topics}, the dual formulation simplifies to 
\begin{equation*}
\W_1(\mu,\nu) = \sup_f ~
\left\{ \int f(x) \ud \mu(x) - \int f(y) \ud \nu(y)
:~ \|f\|_{\op{Lip}}\leq 1 \right\}
\end{equation*}
where $\|f\|_{\op{Lip}}\coloneqq\sup_{x\neq y}\frac{|f(x)-f(y)|}{d(x,y)}$ is the Lipschitz constant of the function $f$ with respect to the metric~$d$. 

\paragraph{Quadratic Cost Function, $p=2$.}\label{sec:quadratic-cost}
Consider the optimization problem~\eqref{eq:Kantorovich-form} with $\calX = \calY=\real^n$ and quadratic cost function $c(x,y)=\frac{1}{2}\norm{x-y}_2^2$. For this special case, the optimization problem can be rewritten as
\begin{align}
\inf_{\pi\in\Pi(\mu,\nu)}\, I(\pi)
&= \inf_{\pi \in \Pi(\mu,\nu)}~\int \frac{1}{2}\norm{x-y}^2 \ud \pi(x,y) \nonumber\\
&= \frac{1}{2}\int \norm{x}^2 \ud \mu(x)
+ \frac{1}{2}\int \norm{y}^2 \ud \nu(y)
- \sup_{\pi\in\Pi(\mu,\nu)}~\int \langle x,y\rangle \ud \pi(x,y) . \label{eq:primal-quad-transform}
\end{align}
Since the first two terms remain constant for all $\pi\in \Pi(\mu,\nu)$, the primal problem~\eqref{eq:Kantorovich-form} is equivalent to
\begin{align}
&\sup_{\pi\in\Pi(\mu,\nu)}~\int \langle x,y\rangle \ud \pi(x,y) .
\label{eq:primal-form-quadratic}
\end{align}
Similarly, with the changes of variables $\bar{f}(x) = \frac{1}{2}\norm{x}_2^2-f(x)$ and $\bar{g}(y) = \frac{1}{2}\norm{y}_2^2-g(y)$, the corresponding dual problem~\eqref{eq:dual-form} can be reformulated as
\begin{align}
\sup_{(f,g)\in\calC(c)}\, J(f,g)
&= \sup_{(f,g)\in\calC(c)}~\int f(x)\ud \mu(x) + \int g(y)\ud \nu(y) \nonumber\\
&= \frac{1}{2}\int \norm{x}^2 \ud \mu(x)
+ \frac{1}{2}\int \norm{y}^2 \ud \nu(y) 
- \inf_{(\bar{f},\bar{g}) \in \overline{\calC} }~\int \bar{f}(x)\ud \mu(x) + \int \bar{g}(y)\ud \nu(y)
\nonumber
\end{align}
where $\overline{\calC}$ is the set of all measurable functions $(f,g) \in L^1(\mu) \times L^1(\nu)$ that satisfy the constraint $f(x) + g(y) \geq \langle x, y\rangle$ for $\mu$-almost all $x \in \real^n$ and $\nu$-almost all $y\in \real^n$; i.e., 
\begin{align*}
\overline{\calC} &\coloneqq \bigl\{(\bar{f},\bar{g})\in L^1(\mu)\times L^1(\nu):~\bar{f}(x)+\bar{g}(y)\geq \langle x, y\rangle~\ud \mu \otimes \ud \nu~\text{a.e.}\bigr\}.
\end{align*} 
Note that $(f,g)\in\calC(c)$ is equivalent to $(\bar{f},\bar{g})\in\overline{\calC}$. 
We use the bar notation to reflect the change in variable and the reversal of the inequality sign compared to the definition $\calC(\cdot)$ in \eqref{eq:def-Phic-set}. 
Similarly, since the first two terms remain constant, the dual problem~\eqref{eq:dual-form} is equivalent to
\begin{align}
\inf_{(f,g) \in \overline{\calC} }~\underbrace{\int f(x)\ud \mu(x) + \int g(y)\ud \nu(y)}_{\bar{J}(f,g)}
\label{eq:dual-form-quadratic}
\end{align}
The following result is known for the quadratic cost setting \cite[Theorems 2.9~and 2.12]{villani2003topics}.

\begin{theorem}\label{thm:Brenier}
	Consider the optimal transportation problem for quadratic cost function where the primal problem is defined as~\eqref{eq:primal-form-quadratic} and its dual formulations defined as~\eqref{eq:dual-form-quadratic}. Assume $X$ and $Y$ have finite second order moments. Then 
	\begin{enumerate}
		\item There exists a pair $(f,f^\star)$, where $f$ is a lower semi-continuous proper convex function and $f^\star$ is its convex conjugate, that minimizes the the dual optimization problem~\eqref{eq:dual-form-quadratic}.
		
		\item (Knott-Smith optimality criterion) $\pi \in \Pi(\mu,\nu)$ is optimal for the primal problem~\eqref{eq:primal-form-quadratic} iff there exists a lower semi-continuous convex function $f$ such that $\op{Supp}(\pi) \subset \op{Graph}(\partial f)$, or equivalently $y \in \partial f(x)$ for all $(x,y) \in \op{Supp}(\pi)$. Moreover, the pair $(f,f^\star)$ is the minimizer of the dual problem~\eqref{eq:dual-form-quadratic}.
		\item (Brenier's theorem) If $\mu$ admits a density with respect to Lebesgue measure, there exists a unique optimal transport map between $\mu$ and $\nu$. The optimal transport map is given by $T(x) = \nabla f(x)$ for $\ud \mu$-almost all $x$ where $f$ is a convex function. The convex function $f$ minimizes the dual formulation~\eqref{eq:dual-form-quadratic}. 
	\end{enumerate}
\end{theorem}

\begin{remark}
	Note that because of~\eqref{eq:primal-quad-transform} and duality, the following relationship holds,
	\begin{equation*}
	\W_2^2(\mu,\nu) = 
	\frac{1}{2} \int \norm{x}^2 \ud \mu(x) + \int \norm{y}^2\ud \nu(y)
	- \inf_{(f,g) \in \overline{\calC} }~ \bar{J}(f,g). 
	\end{equation*}
\end{remark}

\section{Proofs}
\subsection{Proof of Theorem \ref{prop:W2F-metric}}\label{apdx:prop-W2F-metric}
By the definition of the approximate metric~\eqref{eq:W2F}, and the assumption $\W_{2,\calF}(\mu,\nu)=0$, it follows that \[\inf_{\theta \in \Theta} \tilde{J}_{\mu,\nu}(\theta)=\frac{1}{2}\int \norm{x}^2 \ud \mu + \frac{1}{2}\int \norm{y}^2\ud \nu\]
The minimum is achieved for all $\theta_0 \in \Theta_0$ because \[\tilde{J}_{\mu,\nu}(\theta_0)=\bar{J}_{\mu,\nu}(\frac{1}{2}\norm{\cdot}^2,\frac{1}{2}\norm{\cdot}^2) =\frac{1}{2}\int \norm{x}^2 \ud \mu + \frac{1}{2}\int \norm{y}^2\ud \nu.\] 
{ By the first-order optimality condition~\autoref{prop:J-derivative}, all the directional derivatives are zero for all $\theta \in \Theta_0$. Therefore, the result follows.
}

\subsection{Proof of Theorem \ref{prop:W2F-approximability}}\label{apdx:prop-W2F-approximability}
Recall the definitions 
\begin{align}
J_{\mu,\nu}(f,f^\star) &= \frac{1}{2}\int \norm{x}^2\ud \mu(x) + \frac{1}{2}\int \norm{y}^2\ud \nu(y)- \int f(x)\ud \mu - \int f^\star(y) \ud \nu \label{eq:J-def}\\
\W_{2}(\mu,\nu) &= \sup_{f \in \cvx(\calX)}~J^{1/2}_{\mu,\nu}(f) \nonumber
\end{align}

\begin{enumerate}[(i)]
\item By definition, for all $\lambda \in \nabla \calF \# \mu$ there exists $f \in \calF$ and a measurable map $T$ such that $T(x) \in \partial f(x)$ and $\lambda = T\#\mu$. Then, consider the joint distribution $\ud \pi(x,y) = \ud \mu(x)\delta_{y=T(x)}$. The marginals of $\pi$ are equal to $\mu$ and $\lambda$. Also, for all $(x,y) \in \text{supp}(\pi)$ we have $y=T(x)\in \partial f(x)$. 
Therefore, by~\autoref{thm:Brenier-main}, $\pi$ is the optimal coupling between $\mu$ and $\lambda$ and $f$ is the optimal potential function that optimizes the dual problem. Because $f \in \calF$, the restriction to $\calF$ does not change the value of the exact problem. Therefore, 
\begin{equation}\label{eq:W2F-W2-identity}
\W_2(\mu,\lambda) = \W_{2,\calF}(\mu,\lambda)\quad \forall \lambda \in \nabla \calF \# \mu
\end{equation}

\item
For all $\lambda \in \nabla \calF \# \mu$ we have
\begin{align*}
\W_2(\mu,\nu)\leq \W_2(\mu,\lambda) + \W_2(\lambda ,\nu)
=\W_{2,\calF}(\mu,\lambda) + \W_2(\lambda ,\nu)
\end{align*}
where the first line follows from the triangle inequality of $\W_2$, and the second line follows from the identity~\eqref{eq:W2F-W2-identity}. 
Next, we provide upper-bound for $\W_{2,\calF}(\mu,\lambda)$ in terms of $\W_{2,\calF}^2(\mu,\nu)$.
\begin{align*}
\W_{2,\calF}^2(\mu,\lambda)&=\sup_{f \in \calF}~J_{\mu,\lambda}(f,f^\star)
\\
&=\sup_{f \in \calF}~\left[J_{\mu,\nu}(f,f^\star) + \left(\int (\frac{1}{2}\norm{y}^2-f^\star(y)) \ud \lambda(y) - \int (\frac{1}{2}\norm{y}^2-f^\star(y)) \ud \nu(y)\right)\right] \\
&\leq\W^2_{2,\calF}(\mu,\nu) + \sup_{f \in \frac{1}{2}\norm{\cdot}^2-\calF^\star}\left[\int f \ud \lambda - \int f \ud \nu\right]\\
&\leq \W^2_{2,\calF}(\mu,\nu) + c\W_2(\lambda ,\nu)
\end{align*}
where the last inequality follows assumption $\|x-\nabla f^\star(x)\|\leq c_1\|x\|+c_2$ and  \cite[Proposition 1]{polyanskiy2016wasserstein} where $c=(\frac{c_1}{2}\sigma_\nu+\frac{c_1}{2}\sigma_\lambda +c_2)$. Using this result,
\begin{align*}
\W_2(\mu,\nu)&\leq \left[ \W^2_{2,\calF}(\mu,\nu) + c\W_2(\lambda ,\nu)\right]^{1/2} + \W_2(\lambda ,\nu),\quad \forall \lambda \in \nabla \calF \#\mu
\end{align*}
Choosing $\lambda = \text{Proj}(\nu; \nabla \calF \#\mu)$ concludes the result. 
\end{enumerate}

\subsection{Proof of Theorem \ref{prop:generalization}}\label{apdx:prop-generalization}
Denote by $\mathcal{R}_N(\calF,\mu)$ the Rademacher complexity of the function class $\calF$ with respect to $\mu$ for sample size $N$, defined as
\begin{equation*}
R_N(\calF,\mu) \coloneqq \frac{1}{N} \Expect \left[\sup_{f \in \calF}~\sum_{i=1}^N f(X^i)\xi^i \right],
\end{equation*}
where $X^1,\cdots,X^N$ are $N$ i.i.d.~samples from $\mu$, and $\xi^1,\cdots,\xi^N$ are independent Rademacher random variables (taking $+1$ or $-1$, each with probability $1/2$). Here the expectation is over both $\{X^i\}_{i=1}^N$ and the Rademacher random variables $\{\xi^i\}_{i=1}^N$. 

By definition~\eqref{eq:W2F} of $\W_{2,\calF}$ and the notation $J_{\mu,\nu}(f,f^\star)$ defined in~\eqref{eq:J-def}, we have
\begin{align*}
&\left|\W_{2,\calF}^2(\mu^{(N)},\nu^{(N)}) - \W^2_{2,\calF}(\mu,\nu)\right| = \left| \sup_{f\in \calF}~{J}_{\mu,\nu}(f,f^\star) - \sup_{f\in\calF}~{J}_{\mu^{(N)},\nu^{(N)}}(f,f^\star) \right|\\
\quad &\leq \sup_{f\in \frac{1}{2}\norm{\cdot}^2-\calF}~\left|\int f \ud \mu^{(N)} - \int f \ud \mu\right| + \sup_{f\in \frac{1}{2}\norm{\cdot}^2-\calF^\star}~\left|\int f \ud \nu^{(N)} - \int f \ud \nu\right|
\end{align*}
Taking the expectation and using the Rademacher bound concludes the result.

\subsection{Proof of Theorem \ref{prop:J-derivative}} \label{apdx:J-derivative}
	The analysis is similar to \citet{chartrand2009gradient}, but the derivative is computed with respect to the function, not the parameter. 
	Note that
	\begin{align*}
	\nabla_\theta \tilde{J}_{\mu,\nu}(\theta) 
	= \nabla_\theta \left[\int f(x;\theta)\ud \mu(x) + \int f^\star(y;\theta) \ud \nu(y)\right]
	=\nabla_\theta \int f(x;\theta)\ud \mu(x) + \nabla_\theta \int f^\star(y;\theta) \ud \nu(y).
	\end{align*}
	We will show 
	\begin{align}
	\nabla_\theta \int f(x;\theta)\ud \mu(x) &= \int \nabla_\theta f(x;\theta)\ud \mu(x)\label{eq:nabla-theta-f},
	\\ \nabla_\theta \int f^\star(y;\theta) \ud \nu(y)&=\int -\nabla_\theta f(\nabla_y f^\star(y;\theta);\theta)\ud \nu(y).\label{eq:nabla-theta-f-star}
	\end{align}
	To prove~\eqref{eq:nabla-theta-f}, it is sufficient to show
	\begin{equation*}
	\lim_{\theta \to \theta_0}\int \frac{f(x;\theta)-f(x;\theta_0) - (\theta-\theta_0)^\top \nabla_\theta f(x;\theta_0)}{\norm{\theta - \theta_0}_2}\ud \mu(x) = 0.
	\end{equation*}
	By~\autoref{assump}, the function $f$ is differentiable with respect to $\theta$. Hence the limit of the inside of the integral is equal to $0$. Also, inside the integral is bounded by $2L(\theta)$, because $\|\nabla_\theta f(x;\theta_0)\|_2<L(\theta)$ and $|f(x;\theta)-f(x;\theta_0)|\leq L(\theta)\|\theta - \theta_0\|_2$. Therefore, the dominated convergence theorem (DCT) is applicable, concluding~\eqref{eq:nabla-theta-f}.
	
	Proving~\eqref{eq:nabla-theta-f-star} is equivalent to show
	\begin{equation*}
	\lim_{t \to 0} \int \frac{f^\star(y ;\theta_0+ut) - f^\star(y;\theta_0) + u^\top \nabla_\theta f(\nabla_y f^*(y;\theta_0);\theta_0)}{t}\ud \nu(y)= 0
	\end{equation*}
	for all directions $u$ in which $\theta$ is varied. First, we show
	\begin{equation*}
	\lim_{t \to 0} \frac{f^\star(y ;\theta_0+ut) - f^\star(y;\theta_0) - u^\top \nabla_\theta f(\nabla_y f^*(y;\theta_0);\theta_0)}{t}= 0
	\end{equation*}
	for $\ud \nu$-almost everywhere $y$.
	Note that $f^\star(y;\theta)$ is a convex function of $y$ and hence differentiable almost everywhere with respect to $y$. Fix $\theta_0$, and let $y$ be a point such that $\nabla_y f^\star(y;\theta_0)$ exists. Let $x_0 = \nabla_y f^\star(y;\theta_0)$ and $x_t \in \partial_y f^\star(y;\theta_0+ut)$. Then we have the following inequality
	\begin{align*}
	f^\star(y ;\theta_0+ut) - f^\star(y;\theta_0) &=\sup_x~(\langle x, y\rangle - f(x;\theta_0+ut)) - \sup_x~(\langle x, y\rangle - f(x;\theta_0))\\&\geq \langle x_0 , y\rangle - f(x_0;\theta_0+ut) - ( \langle x_0, y\rangle - f(x_0;\theta_0))\\
	&= -(f(x_0;\theta_0+ut)-f(x_0;\theta_0)).
	\end{align*}
	Taking the limit as $t \to 0$ proves 
	\begin{equation*}
	\limsup_{t \to 0} \frac{f^\star(y ;\theta_0+ut) - f^\star(y;\theta_0) +t u^\top \nabla_\theta f(x_0;\theta_0)}{t} \geq 0.
	\end{equation*}
	It remains to prove the inequality in the other direction. 
	Extract a convergent subsequence $\{x_{t_k}\}_{k=1}^\infty$ form $x_t \in \partial_y f^\star(y;\theta_0+ut)$ that converges to $x_0$. Such a subsequence exists, because the supoprt of $\nu$ is compact. Then 
	\begin{align*}
	f^\star(y ;\theta_0+ut_k) - f^\star(y;\theta_0) &=\sup_x~(\langle x, y\rangle - f(x;\theta_0+ut_k)) - \sup_x~(\langle x, y\rangle - f(x;\theta_0))\\&\leq \langle x_{t_k} , y \rangle - f(x;\theta_0+ut_k) - ( \langle x_{t_k} , y \rangle - f(x_{t_k};\theta_0))\\
	&= -(f(x_{t_k};\theta_0+ut_k)-f(x_{t_k};\theta_0)).
	\end{align*}
	Taking the limit as $k\to \infty$, using $\theta_{t_k} \to \theta_0$, $x_{t_k} \to x_0$, differentiability of $f(x;\theta) $, and $\nabla_\theta f(x;\theta)$ being continuous with respect to $x$, we conclude
	\begin{align*}
	\liminf_{k \to \infty} \frac{f^\star(y ;\theta+ut_k) - f^\star(y;\theta_0) + t_ku^\top\nabla_\theta f(x_0;\theta_0)}{t_k} &\leq 0.
	\end{align*}
	Putting these results together we get
	\begin{equation*}
	\lim_{t \to 0} \frac{f^\star(y ;\theta_0+ut) - f^\star(y;\theta_0) - u^\top \nabla_\theta f(\nabla_y f^*(y;\theta_0);\theta_0)}{t}= 0
	\end{equation*}
	where we used $x_0=\nabla f^\star(y;\theta_0)$. Note that, through this procedure, we can conclude the upper-bound,
	\begin{equation*}
	| \frac{f^\star(y ;\theta_0+ut) - f^\star(y;\theta_0) - u^\top \nabla_\theta f(\nabla_y f^*(y;\theta_0),\theta_0)}{t}| \leq 2L(\theta)\|u\|_2.
	\end{equation*}
	Therefore, by DCT, ~\eqref{eq:nabla-theta-f-star} follows.

\section{Duality of Conic Linear Programs for Optimal Transport}
\label{sec:restr-orig}
In this section, we formalize a unified language towards understanding the set restrictions in function classes and classes of probability distribution that arise in primal and dual approximations of the Wasserstein distance. In part, we borrow from the conic duality theory for infinite-dimensional linear programs but also examine properties of the optimal solution and optimal value from the point of view of the optimal transport theory.

\subsection{A Partial Order}\label{sec:order}
Suppose $\calX$ is a Polish space and consider any function class $\mathcal{F}\subset C_b(\calX)$. Let us begin by defining a {\em preorder} $\preceq_\calF$ (a reflexive and transitive relation) on the set of finite measures $M(\calX)$ according to
\begin{align*}
\mu \preceq_\calF \tilde{\mu} \quad \Leftrightarrow\quad \int f(x) \ud \mu(x) \leq \int f(x) \ud \tilde{\mu}(x),\quad \forall f \in \calF
\end{align*}
for any $\mu,\tilde{\mu} \in M(\calX)$. Given this preorder, we define {\em an equivalence relation} on $M(\calX)$ as 
\begin{equation*}
	\mu \equiv_\calF \tilde{\mu}
	\quad \Leftrightarrow\quad 
	\mu \preceq_\calF \tilde{\mu} ~,~ \tilde{\mu} \preceq_\calF \mu	
	\quad \Leftrightarrow\quad 
	\int f(x) \ud \mu(x) = \int f(x) \ud \tilde{\mu}(x),\quad \forall f \in \calF.
\end{equation*}
\begin{remark}\label{rem:cone-span}
From the definitions, it is easy to see that 
\begin{itemize}
\item $\equiv_\calF$ is the same as $\equiv_{\op{span}(\calF)}$, 
where $\op{span}(\calF) \coloneqq \bigl\{ \sum_{i=1}^k \lambda_i f_i:~ k\in\mathbb{N},~ f_i\in\calF,~ \lambda_i \in \mathbb{R}\bigr\}$, and, 
\item $\preceq_\calF$ is the same as $\preceq_{\op{cone}(\calF)}$, 
where $\op{cone}(\calF) \coloneqq \bigl\{ \sum_{i=1}^k \lambda_i f_i:~ k\in\mathbb{N},~ f_i\in\calF,~ \lambda_i \in \mathbb{R}_+\bigr\}$.
\end{itemize}
\end{remark}
\begin{remark}\label{rem:reflection-symmetry}
Consider the case where $\calF$ is symmetric with respect to reflection, i.e., if $f \in \calF$, then $-f \in \calF$. Then, the partial order relationship $\preceq_\calF$ is equal to the equivalence relationship $\equiv_\calF$, i.e., 
\begin{equation*}
\mu \preceq_\calF \tilde{\mu} ~\Leftrightarrow~\tilde{\mu} \preceq_\calF {\mu} ~\Leftrightarrow~\mu \equiv_\calF \tilde{\mu}.
\end{equation*}
\end{remark}
Let $[\mu]_\calF$ denote the equivalence class of $\mu$ with respect to the function class $\calF$. 
The {\em quotient space}, namely 
\[M(\calX)/(\equiv_\calF)\;\coloneqq\;\bigl\{[\mu]_\calF:~ \mu \in M(\calX) \bigr\},\] is defined to be the set of all equivalence classes constructed with the equivalence relation~$\equiv_\calF$. 
The preorder notation on $M(\calX)$ can be overloaded to a {\em partial order} (an antisymmetric preorder) on $M(\calX)/(\equiv_\calF)$ where we define 
\begin{align*}
	[\mu]_\calF \preceq_\calF [\tilde{\mu}]_\calF
	\quad \Leftrightarrow\quad
	\mu \preceq_\calF \tilde{\mu}. 
\end{align*}
We denote the inverse by $\succeq_\calF$. 

A function class $\calF$ is {\em separating} if $\mu \equiv_\calF \nu$ implies $\mu=\nu$; i.e., $[\mu]_\calF$ is a singleton for all $\mu$. For example, 
\begin{itemize}
\item Consider $\calF$ to to be the set of all convex quadratic functions. Then, $[\mu]_\calF$ is the set of all probability measures with the same mean and covariance as $\mu$. 

\item Consider $\calX$ to be a compact subset of the Euclidean space and consider $\calF$ to be class of all polynomials of degree at most $k$ on $\calX$. Then, $[\mu]_\calF$ is the set of all probability distributions supported on $\calX$ that have the same set of first $k$ moments that match those of~$\mu$. 

\item Consider $\calF$ to be $C_b(\mathcal{X})$. Then $\calF$ is separating. 

\end{itemize}
In the rest of this section, we establish a framework for how existing notions in the context of Kantorovich duality can be extended to yield a new duality framework according to the preorder we define.

\subsection{The Couplings}
For any measure $\pi \in M(\calX \times \calY)$, let $\pi_x$ and $\pi_y$ denote its marginals on $\calX$ and $\calY$ respectively. By definition, they satisfy the following identities,
\begin{align*}
\int_{\calX\times \calY} f(x)\ud \pi(x,y) &= \int_\calY f(y)\ud \pi_x(x),\quad \forall f \in C_b(\calX),\\
\int_{\calX\times \calY} g(y)\ud \pi(x,y) &= \int_\calY g(y)\ud \pi_y(y),\quad \forall g \in C_b(\calY).
\end{align*}
\begin{definition}
For any two classes of functions $\calF$ and $\calG$, with $\calK\coloneqq \calF\times \calG$, and any two measures $\mu \in M(\calX)$ and $\nu \in M(\calY)$, define the following sets of joint distributions, 
\begin{equation*}
\Pi_\equiv^\calK(\mu,\nu)
\coloneqq \bigl\{\pi \in M(\calX\times \calY);~\pi_x\equiv_\calF \mu,~\pi_y\equiv_\calF \nu \bigr\},
\end{equation*}
and,
\begin{align}
\Pi_\succeq^\calK (\mu,\nu)
&\coloneqq \bigl\{\pi\in M(\calX\times \calY):~\pi_x\succeq_\calF \mu,~\pi_y\succeq_\calG \nu \bigr\}, \label{eq:def-Pi-succeq}\\
\Pi_\preceq^\calK (\mu,\nu)
&\coloneqq \bigl\{\pi\in M(\calX\times \calY):~\pi_x\preceq_\calF \mu,~\pi_y\preceq_\calG\nu \bigr\}. \label{eq:def-Pi-preceq}
\end{align}
To simplify the notation, we use $\Pi_\succeq ([\mu]_\calF,[\nu]_\calG)$ instead of $\Pi_\succeq^\calK (\mu,\nu)$ whenever clear from the context.
\end{definition}
By definition, $\pi \in \Pi^\calK_\equiv(\mu,\nu)$ if and only if
\begin{subequations}
\begin{align}
\int f(x)\ud \pi(x,y)&= \int f(x)\ud \mu(x),\quad \forall f \in \calF,\label{eq:couplings-equiv-def-x}\\ \int g(y)\ud \pi(x,y)&=\int g(y)\ud \nu(y),\quad \forall g \in \calG.\label{eq:couplings-equiv-def-y}
\end{align}
\end{subequations}
Moreover, if $\calF$ and $\calG$ are separating (for example, if $\calF=C_b(\calX)$ and $\calG=C_b(\calY)$) then $\Pi^\calK_\equiv(\mu,\nu)=\Pi(\mu,\nu)$ is the set of joint distributions with marginals $\mu$ and~$\nu$. In general, $\Pi^\calK_\equiv(\mu,\nu)$ could be larger than $\Pi(\mu,\nu)$. In fact, we can establish the following relationship. 
\begin{lemma}\label{lem:Pi-union}
Given two distributions $\mu\in M(\mathcal{X})$ and $\nu\in M(\mathcal{Y})$ and two function classes $\calF\subseteq C_b(\calX)$ and $\calG\in C_b(\calY)$, 
with the above notation, we have
\begin{align*}
\Pi_\equiv^\calK(\mu,\nu) &~=~ \bigcup \bigl\{ \Pi(\tilde{\mu},\tilde{\nu}) :~ \tilde{\mu}\equiv_\calF \mu,\,\tilde{\nu}\equiv_\calG \nu \bigr\}, \\
\Pi_\succeq^\calK(\mu,\nu) &~=~ \bigcup \bigl\{ \Pi(\tilde{\mu},\tilde{\nu}) :~ \tilde{\mu}\succeq_\calF \mu,\,\tilde{\nu}\succeq_\calG \nu \bigr\}, \\
\Pi_\preceq^\calK(\mu,\nu) &~=~ \bigcup \bigl\{ \Pi(\tilde{\mu},\tilde{\nu}) :~ \tilde{\mu}\preceq_\calF \mu,\,\tilde{\nu}\preceq_\calG \nu \bigr\}.
\end{align*}
\end{lemma}
\begin{proof}Let us prove the first assertion. The rest are similar. 

We first establish the backward inclusion. Let's take an arbitrary member of the right-hand side: take any $\tilde{\mu}\in [\mu]_\calF$ and any $\tilde{\nu}\in[\nu]_\calG$ and consider any $\tilde{\pi}\in \Pi(\tilde{\mu},\tilde{\nu})$. We need to establish~\eqref{eq:couplings-equiv-def-x}-\eqref{eq:couplings-equiv-def-y} for $\tilde{\pi}$ which is easy using the three aforementioned memberships.

For the forward inclusion, consider a distribution $\pi$ on $\mathcal{X}\times \mathcal{Y}$ that satisfies~\eqref{eq:couplings-equiv-def-x}-\eqref{eq:couplings-equiv-def-y}. Let $\tilde{\mu}$ and $\tilde{\nu}$ be marginals of $\pi$. Then by definition, $\int f(x)\tilde{\mu}(x)=\int f(x) \ud \pi(x,y) = \int f(x)\ud \mu(x)$ for all $f\in\calF$, hence $\tilde{\mu} \in [\mu]_\calF$ and similarly $\tilde{\nu}\in [\nu]_\calG$. This proves the forward inclusion. 
\end{proof}

\subsection{The Two Dual Optimization Problems}\label{sec:opts-ineq-version}
Given $\mu$, $\nu$, $\calF$, and $\calG$, define $\calK\coloneqq \calF \times \calG$. For notational simplicity, and as it is clear from the context, we will omit the dependence on $\mu$ and $\nu$ throughout this section. In parallel with~\eqref{eq:Kantorovich-form}, we define a {\em restricted optimal transportation problem} as
\begin{equation}\label{eq:primal-constrained}
	\inf_{\pi \in \Pi_\succeq^\calK(\mu,\nu)} 
	~ \underbrace{\int c(x,y)\ud \pi(x,y)}_{I(\pi)}.
\end{equation}
We also define a problem in parallel with the original Kantorovich's dual problem in~\eqref{eq:dual-form} as
\begin{equation}\label{eq:dual-constrained}
	\sup_{(f,g) \in \calC(c) \cap \calK} ~ \underbrace{\int f(x)\ud \mu(x) + \int g(y)\ud \nu(y)}_{J(f,g)}
\end{equation}
where the constraint set $\calC(c)$ is defined in~\eqref{eq:def-Phic-set}. 
\begin{proposition}[Weak Duality]
For~\eqref{eq:primal-constrained} and~\eqref{eq:dual-constrained}, we have 
\begin{equation*}
	I(\pi)\geq J(f,g)
\end{equation*}
for all $\pi \in \Pi_\succeq^\calK(\mu,\nu)$ and all $(f,g) \in \calC(c) \cap \calK$.
\end{proposition}
\begin{proof}
Since $\pi \in \Pi_\succeq^\calK(\mu,\nu)$, $f \in \calF$, and $g\in\calG$, we get from~\eqref{eq:def-Pi-succeq} that 
$\int f(x) \ud \pi(x,y) \geq \int f(x) \ud \mu(x)$ and 
$\int g(y) \ud \pi(x,y) \geq \int g(y) \ud \nu(y)$. 
Since $(f,g)\in \calC(c)$, defined in \eqref{eq:def-Phic-set}, we have $\int (f(x)+g(y)) \ud \pi(x,y) \leq \int c(x,y)\ud \pi(x,y) = I(\pi)$. 
Putting these inequalities together establishes the claim. 
\end{proof}

The following theorem shows that the duality gap is zero if $\calF$ and $\calG$ are convex cones. This can be viewed as the generalization of the Kantorovich's duality in \autoref{thm:duality} for the case of restriction to convex cones. The proof appears in \autoref{apdx:duality-proof}. 

\begin{theorem}[Strong Duality]\label{thm:duality-constrained}
	Consider the optimization problems~\eqref{eq:primal-constrained} and \eqref{eq:dual-constrained} where $\calF$ and $\calG$ are {\em convex cones} as subset of $C_b(\calX)$ and $C_b(\calY)$, respectively. Assume the cost function $c$ is continuous, the sets $\calX$ and $\calY$ are compact, and there exists a pair $(f_0,g_0)\in \calK \coloneqq \calF\times \calG$ such that $f_0(x)+g_0(y)< c(x,y)$ for all $(x,y)\in \calX \times \calY$. Then,
	\begin{equation*}
		\inf_{\pi \in \Pi_\succeq^\calK(\mu,\nu)}~ I(\pi) ~= \sup_{(f,g) \in \calC(c) \cap \calK}~J(f,g).
	\end{equation*}
\end{theorem}
\begin{remark}
In general, for any $\calF$ and $\calG$, and $\calK = \calF \times \calG$, consider $\op{cone}(\calK)=\op{cone}(\calF) \times \op{cone}(\calG) $. Then, the strong duality of \autoref{thm:duality-constrained} implies
\begin{equation*}
\inf_{\pi \in \Pi_\succeq^{\calK}(\mu,\nu)}~ I(\pi) ~= \inf_{\pi \in \Pi_\succeq^{\op{cone}(\calK)}(\mu,\nu)}~ I(\pi) ~= \sup_{(f,g) \in \calC(c) \cap \op{cone}(\calK)}~J(f,g)~\geq \sup_{(f,g) \in \calC(c) \cap \calK}~J(f,g)
\end{equation*}
where the conclusion from \autoref{rem:cone-span} is used. 
Similarly, consider $\op{span}(\calK)=\op{span}(\calF) \times \op{span}(\calG) $. Then, the strong duality of \autoref{thm:duality-constrained} implies
\begin{equation*}
\inf_{\pi \in \Pi_\equiv^{\calK}(\mu,\nu)}~ I(\pi) ~= \inf_{\pi \in \Pi_\equiv^{\op{span}(\calK)}(\mu,\nu)}~ I(\pi) ~= \inf_{\pi \in \Pi_\succeq^{\op{span}(\calK)}(\mu,\nu)}~ I(\pi) ~= \sup_{(f,g) \in \calC(c) \cap \op{span}(\calK)}~J(f,g). 
\end{equation*}
where the conclusions from \autoref{rem:cone-span} and \autoref{rem:reflection-symmetry} are used. 
\end{remark}

\subsection{The Optimal Transport Map}\label{apdx:conic-qaudratic}
Consider the case where $\calX$ and $\calY$ are compact subsets of $\real^n$. 
In~\autoref{sec:quadratic-cost} we derived {\em equivalent} optimization problems \eqref{eq:primal-form-quadratic} and \eqref{eq:dual-form-quadratic} for the original dual pair of problems \eqref{eq:Kantorovich-form} and \eqref{eq:dual-form}, respectively. This was done through changing the cost function from $c_1(x,y) =\frac{1}{2} \norm{x-y}^2$ to $c_2(x,y) =- \langle x, y\rangle$ and updating $\calC(c_1)$ to $\overline{\calC}$ for $c_2$. However, the same equivalent transformation is not straightforward when working with restricted problems~\eqref{eq:primal-constrained} and~\eqref{eq:dual-constrained}. Nonetheless, we consider the following two optimization problems, 
\begin{align}
	&\sup_{\pi \in \Pi_\preceq^\calK(\mu,\nu)} ~\int \langle x, y\rangle\ud\pi(x,y)\label{eq:primal-quad-constrained}\\
	&\inf_{(f,g)\in \overline{\calC}\cap \calK} ~\int f(x)\ud \mu(x) + \int g(y)\ud \nu(y)\label{eq:dual-quad-constrained}
\end{align} 
where the constraint set $\Pi_\preceq^\calK(\mu,\nu)$ is given in~\eqref{eq:def-Pi-preceq} and 
\begin{equation*}
\overline{\calC} \eqqcolon \{(f,g) \in C_b(\calX) \times C_b(\calY);~f(x)+g(y)\geq \langle x,y \rangle\}.
\end{equation*} 
\begin{theorem}\label{thm:restricted-duality}
Under the assumptions of \autoref{thm:duality-constrained}, the optimal values of \eqref{eq:primal-quad-constrained} and \eqref{eq:dual-quad-constrained} are equal. 
\end{theorem}
\begin{proof}
The proof is application of the strong duality in \autoref{thm:duality-constrained}. 
\begin{align*}
\sup_{\pi \in \Pi_\preceq^\calK(\mu,\nu)} ~\int \langle x, y\rangle\ud\pi(x,y)~&=
-\inf_{\pi \in \Pi_\succeq^{-\calK}(\mu,\nu)} ~\int (-\langle x, y\rangle)\ud\pi(x,y)\\
&=	 - \sup_{(f,g)\in{\calC(c_2)}\cap (-\calK)} ~\int f(x)\ud \mu(x) + \int g(y)\ud \nu(y)\\
&= - \sup_{(-f,-g)\in (-\calC(c_2))\cap \calK} -\left(~\int (-f(x))\ud \mu(x) + \int (-g(y))\ud \nu(y)\right)\\
&=\inf_{(f,g)\in \overline{\calC}\cap \calK} ~\int f(x)\ud \mu(x) + \int g(y)\ud \nu(y)
\end{align*}
where we used $-\calC(c_2) = \overline{\calC}$, because $-f(x) -g(y)\leq -\langle x,y \rangle$ is equivalent to $f(x)+g(y)\geq \langle x,y \rangle$.
\end{proof}
\autoref{thm:Brenier-restricted} is analogous to \autoref{thm:Brenier}. The proof appears in~\autoref{apdx:proof-Brenier-restricted}.

\begin{theorem}\label{thm:Brenier-restricted}
	Assume the conditions of \autoref{thm:Brenier-restricted} hold. 
	Consider the primal and dual problems~\eqref{eq:primal-quad-constrained}-\eqref{eq:dual-quad-constrained}, where duality was established by \autoref{thm:Brenier-restricted}. 
	Assume the supremum in~\eqref{eq:primal-quad-constrained} and the infimum in \eqref{eq:dual-quad-constrained} are attained with $\bar{\pi}$ and $(\bar{f},\bar{g})$ respectively. Then, 
	\begin{enumerate}
		\item 
		\begin{equation}\label{eq:optimality-pi-contrained}
		\begin{aligned}
&\int (\bar{f}(x) + \bar{g}(y))\ud \bar{\pi}(x,y) = \int \bar{f}(x)\ud \mu(x) + \int \bar{g}(y) \ud \nu(y)\quad \op{and}\\
		&(x,y) \in \op{supp}(\pi)~\Leftrightarrow~y \in \partial f^{\star\star}(x),~ x \in \partial g^{\star\star}(y),~ g^{\star}(x)=f(x),~ f^\star(y)=g(y).
		\end{aligned}
		\end{equation}
		\item If $\calF^{\star\star}\subset \calF$ and $\calG^{\star\star}\subset \calG$, then the minimizer pair of~\eqref{eq:dual-quad-constrained} are convex functions. 
		\item If $\calF^{\star\star} \subset \calF$ and $\calG = C_b(\calY)$, then the minimizer pair of ~\eqref{eq:dual-quad-constrained} are of the form $(f,f^\star)$ for a convex function $f \in \calF$ and 
		\begin{equation*}
		(x,y) \in \op{supp}(\pi)~\Leftrightarrow~y \in \partial f(x),~ x \in \partial f^{\star}(y).
		\end{equation*}
		Moreover if $\nu$ admit a Lebesgue density $\nabla f^\star \# \nu \succeq_\calF \nu$, i.e., 
		\begin{equation*}
		\int \tilde{f}(\nabla f^\star(y)) \ud \nu (y) \leq \int \tilde{f}(x)\ud \mu(x),\quad \forall \tilde{f} \in \calF.
		\end{equation*}
	\end{enumerate} 
\end{theorem}

\begin{remark}
The restricted optimal transport problem with $c_1(x,y)=\frac{1}{2}\norm{x-y}^2$ and $c_2(x,y)=-\langle x,y \rangle$ are related when the convex cone $\calF$ and $\calG$ contain the quadratic functions $\{+\frac{1}{2}\|\cdot\|_2^2,-\frac{1}{2}\|\cdot\|_2^2\}$. Then,
\begin{equation*}
\inf_{\pi \in \Pi_{\succeq}^\calK(\mu,\nu)} \int \frac{1}{2}\norm{x-y}^2 \ud \pi(x,y)~= \frac{1}{2}\int \norm{x}^2 \ud\mu(x) + \frac{1}{2} \int \norm{y}^2 \ud \nu(y) - \sup_{\pi \in \Pi_{\preceq}^{-\calK}(\mu,\nu)} \int \langle x,y \rangle \ud \pi(x,y)
\end{equation*}
\end{remark}

\begin{corollary}
Consider the setting of the \autoref{thm:Brenier-restricted}. Then, 
\begin{itemize}
\item Let $\calF_{\op{convex}} \coloneqq \calF \cap \cvx(\calX) \subset \calF$ denote the set of convex functions in $\calF$. 
Then for any convex cone $\calF$
\begin{align*}
\inf_{f\in \calF_{\op{convex}}}~J_{\mu,\nu}(f,f^\star) 
= \inf \left\{ J_{\mu,\nu}(f,g) :~ (f,g)\in \overline{\calC} \cap(\calF \times C_b(\calY)) \right\}
= \sup \left\{ I(\pi):~ \pi \in \Pi^{\calF \times C_b(\calY)}_{\preceq}(\mu,\nu)\right\}
\end{align*}
\item Let $\calF$ be a set of convex functions (not necessarily a convex cone). Then,
\begin{align*}
\inf_{f\in \op{Cone}(\calF)}~J_{\mu,\nu}(f,f^\star) 
&= \inf \left\{ J_{\mu,\nu}(f,g) :~ (f,g)\in \overline{\calC} \cap(\op{Cone}(\calF) \times C_b(\calY)) \right\}\\
&= \sup \left\{ I(\pi) :~ \pi \in \Pi^{\op{Cone}(\calF) \times C_b(\calY)}_{\preceq}(\mu,\nu) \right\}
\end{align*}
\end{itemize}
\end{corollary}

\subsection{Metrics}\label{sec:conic-metric}

Given two function classes $\calF$ and $\calG$ with $\calK \coloneqq \calF \times \calG$, $p\in[1,\infty)$, and $\mu,\nu \in \calP_p(\calX)$, define the following two distances
\begin{equation}
	\W_{\calK,p}^{\prim}(\mu,\nu) 
	\coloneqq \inf_{\pi \in \Pi_{\succeq}^\calK(\mu,\nu)}~\left[\int d(x,y)^p\ud \pi(x,y)\right]^{\frac{1}{p}}
	\label{eq:W2-constrained-primal}
\end{equation}
and
\begin{equation*}
	\W_{\calK,p}^{\dual}(\mu,\nu) \coloneqq 
	\sup_{(f,g) \in \calC(d^p) \cap \calK} ~ \left[ \int f(x)\ud \mu(x) + \int g(y)\ud \nu(y) \right]^{1/p}
\end{equation*}
where $\calC(d^p) = \bigl\{(f,g)\in L^1(\mu)\times L^1(\nu):~f(x)+g(y)\leq d(x,y)^p~\ud \mu \otimes \ud \nu~\text{a.e.}\bigr\}$. 
If the assumptions of the strong duality hold, i.e., $\calF$ and $\calG$ are convex cones, then $\W_{\calK,p}^{\prim}(\mu,\nu) = \W_{\calK,p}^{\dual}(\mu,\nu)$. 

Note that because of the relationship $\Pi_{\succeq}^\calK(\mu,\nu)=\bigcup_{\tilde{\mu}\succeq_\calF\mu,\tilde{\nu}\succeq_\calG \nu} \Pi(\tilde{\mu},\tilde{\nu})$ we conclude
\begin{equation}\label{eq:W2-restricted-inf-identity}
\begin{aligned}
\W_{\calK,p}^{\prim}(\mu,\nu) 
= \inf_{\tilde{\mu}\succeq_\calF\mu,\tilde{\nu}\succeq_\calG \nu} ~\inf_{\pi \in \Pi(\tilde{\mu},\tilde{\nu})}~\left[\int d(x,y)^p\ud \pi(x,y)\right]^{\frac{1}{p}}
=\inf_{\tilde{\mu}\succeq_\calF\mu,\tilde{\nu}\succeq_\calG \nu} \W_p(\tilde{\mu},\tilde{\nu}).
\end{aligned}
\end{equation}

\begin{proposition}\label{prop:metric} 
Consider the definition~\eqref{eq:W2-constrained-primal}. Then for all probability measures $\mu,\nu,\lambda\in \calP_p(\calX)$:
	\begin{enumerate}
		\item $\W_{\calK,p}^{\prim}(\mu,\nu) = 0$ iff $\exists \lambda \in \calP_p(\calX)$ such that $\lambda \succeq_\calF \mu$ and $\lambda \succeq_\calG \nu$ 
		\item $\W^{\prim}_{\calF \times \calG,p}(\mu,\nu)=\W^{\prim}_{\calG \times \calF,p}(\nu,\mu)$
		\item $\W^{\prim}_{\calF \times \calG,p}(\mu,\nu) \leq \W^{\prim}_{\calF \times \mathcal{H},p}(\mu,\lambda) + \W^{\prim}_{\mathcal{H} \times \calG,p}(\lambda,\nu)$
	\end{enumerate}
\end{proposition}
\begin{proof}
	\begin{enumerate}
		\item $\W^\calK_{p}(\mu,\nu) = 0$ implies that there exists a coupling $\pi \in \Pi_{\succeq}^\calK(\mu,\nu)$ which is concentrated on the diagonal $x=y$. Therefore, the marginals are equal, i.e., $\pi_x=\pi_y$. By definition, $\pi_x \succeq_\calF \mu$ and $\pi_y \succeq_\calF \nu$. Therefore, $\lambda=\pi_x=\pi_y$ is the required measure
		\item The symmetry property easily follows from the definition.
		\item The triangle inequality follows from~\eqref{eq:W2-restricted-inf-identity}. For all $\tilde{\mu} \in [\mu]_\calF$, $\tilde{\nu}\in [\nu]_\calG$, and $\tilde{\lambda}$ we have $\W_{p}(\tilde{\mu},\tilde{\nu}) \leq \W_{p}(\tilde{\mu},\tilde{\lambda}) + \W_{p}(\tilde{\lambda},\tilde{\nu})$. Taking the infimum over $\tilde{\mu}\in[\mu]_\calF$, $\tilde{\nu}\in [\nu]_\calG$ and $\tilde{\lambda} \in [\lambda]_\mathcal{H}$ concludes the result.
		 
	\end{enumerate}
\end{proof}
\begin{corollary}
Consider the case where $\calX=\calY$, and $\calF =\calG$ is a linear subspace. Then,
	\begin{enumerate}
	\item $\W_{\calK,p}^{\prim}(\mu,\nu) = 0$ iff $\mu \equiv_\calF \nu $,
	\item $\W_{\calK,p}^{\prim}(\mu,\nu)=\W_{\calK,p}^{\prim}(\nu,\mu)$,
	\item $\W_{\calK,p}^{\prim}(\mu,\nu) \leq \W_{\calK,p}^{\prim}(\mu,\lambda) + \W_{\calK,p}^{\prim}(\lambda,\nu)$.
\end{enumerate}
Therefore, $\W_{\calK,p}^{\prim}$ is a metric on the quotient space $M(\calX)/\equiv_\calF$.
\end{corollary}

\begin{proposition}The dual version $\W_{\calK,p}^{\dual}(\mu,\nu) $ satisfies the result in \autoref{prop:metric} when the strong duality in \autoref{thm:duality-constrained} holds, i.e., $\calF$ and $\calG$ are convex cones. 
\end{proposition}

\subsection{Proof of Theorem~\ref{thm:duality-constrained}}\label{apdx:duality-proof}
	\begin{theorem}[Fenchel-Rockafellar duality] \label{thm:min-max}
	Let $E$ be a normed vector space, $E^\star$ its topological dual space, and $\Theta$,~$\Xi$ two convex functions on $E$ with values in $\real \cup \{+\infty\}$. Let $\Theta^\star$ and $\Xi^\star$ be the Legendre-Fenchel transforms of $\Theta$ and~$\Xi$, respectively. Assume $\exists x_0 \in E$ such that 
	\begin{equation*}
	\Theta(x_0)< +\infty, \quad \Xi(x_0)< +\infty,\quad \Theta \text{ is continuous at } x_0
	\end{equation*}
	Then,
	\begin{equation}\label{eq:min-max}
	\inf_{x\in E} \left[\Theta(x)+\Xi(x)\right] = \max_{y \in E^\star}\left[-\Theta^\star(y) - \Xi^\star(-y)\right]
	\end{equation}
\end{theorem}
\begin{proof}[Proof of \autoref{thm:duality-constrained}]
	The proof is a modification of the proof of \cite[Theorem~1.3~pp 26]{villani2003topics} which is an application of the Fenchel-Rockafellar duality in \autoref{thm:min-max}. 
	Let
	\begin{equation*}
	E=C_b(\calX \times \calY)
	\end{equation*} 
	be the set of all bounded continuous functions on $\calX \times \calY$ equipped with the sup-norm $\|\cdot\|_\infty$. By Riesz's theorem, its topological dual is identified with the space of (regular) Radon measures 
	\begin{equation*}
	E^\star = M(\calX \times \calY)
	\end{equation*}
	normed by total-variation. The linear operation of a dual element $\pi \in E^\star$ on $u \in E$ is defined according to 
	\begin{equation*}
	\pi(u)=\int_{\cal X\times \calY} u(x,y)\ud \pi(x,y)
	\end{equation*}
	Define the functions $\Theta:E \to \real \cup \{+\infty\}$ and $\Xi:E \to \real \cup \{+\infty\}$ as
	\begin{align*}
	\Theta(u) &= \begin{cases}
	0 & \text{if} ~ u(x,y)\leq c(x,y) ~\forall x,y\\
	+\infty & \text{otherwise,}
	\end{cases}\\
	\Xi(u) &= \begin{cases}
	-\int_\calX f(x)\ud \mu(x)- \int_\calY g(y)\ud\nu(y) & \text{if}~\exists (f,g)\in \calF\times \calG~\text{s.t}~u(x,y)=f(x)+g(y)~\forall x,y\\
	+\infty & \text{otherwise,}
	\end{cases}
	\end{align*}
	for all $u \in C_b(\calX\times \calY)$. 
	Note that $\Xi(u)$ is well-defined. If there exists two pairs $(f,g)$ and $(\tilde{f},\tilde{g})$ such that $u(x,y)=f(x) + g(y)=\tilde{f}(x)+\tilde{g}(y)$, then $f(x)-\tilde{f}(x)=\tilde{g}(y)-{g}(y)$. This identity hold for all $(x,y)$ only if $f(x)-\tilde{f}(x)=\tilde{g}(y)-\tilde{g}(y)=c$ is a constant. Hence $f(x)=\tilde{f}(x)+c$ and $g(y)=\tilde{g}(y)-c$. Therefore, $\int f \ud \mu +\int g \ud \nu = \int \tilde{f}\ud \mu + \int \tilde{g}\ud \nu$.
	\medskip
	
	The assumptions of the Fenchel-Rockafellar duality theorem are satisfied: 
	\begin{enumerate}
		\item $\Theta$ is convex because for all $u_1,u_2 \in E$ such that $u_1(x,y)\leq c(x,y)$, $u_2(x,y)\leq c(x,y)$, and for all $\lambda \in[0,1]$, we have \[\lambda u_1(x,y) + (1-\lambda)u_2(x,y)\leq c(x,y).\] 
		\item $\Xi$ is convex because $\forall u_1,u_2 \in E$ such that $u_1(x,y)= f_1(x)+g_1(y)$, $u_2(x,y)=f_2(x)+g_2(y)$ and $\lambda \in[0,1]$ with $f_1,f_2 \in \calF$ and $g_1,g_2\in \calG$, we have 
		\[\lambda u_1(x,y) + (1-\lambda)u_2(x,y) = (\lambda f_1(x) + (1-\lambda)f_2(x)) + (\lambda g_1(y) + (1-\lambda)g_2(y))\]
		Because $\calF$ and $\calG$ are convex sets, $\lambda f_1 + (1-\lambda)f_2 \in \calF$ and $\lambda g_1 + (1-\lambda)g_2 \in \calG$ (Here the assumption that $\calF$ is convex is used). Therefore, 
		\begin{align*}
		\Xi(\lambda u_1 +(1-\lambda u_2)) &= -\int (\lambda f_1 + (1-\lambda) f_2)\ud \mu - \int (\lambda g_1 + (1-\lambda)g_2)\ud \nu \\&=\lambda \Xi(u_1)+ (1-\lambda)\Xi(u_2)
		\end{align*}
		\item According to the Assumption, there exists a feasible pair $(f_0,g_0)\in \calF \times \calG$ such that $f_0(x)+g_0(y)
		< c(x,y)$ (note that the inequality should be strict). Taking $u_0(x,y)=f_0(x)+g_0(y)$, we can see that $\Theta(u_0)=0$ because $u_0(x,y)=f_0(x)+g_0(y)<c(x,y)$. Also $\Xi(u_0)= -\int f_0\ud \mu -\int g_0 \ud \nu < +\infty$. Moreover, $\Theta$ is continuous at $u_0$. Let $\epsilon = \inf_{(x,y)\in\calX\times \calY} [c(x,y)-u_0(x,y)]>0$. Then for all $\tilde{u}\in E$ such that $\|\tilde{u}-u_0\|_\infty <\epsilon$, we have $\tilde{u}(x,y)\leq u_0(x,y)+\epsilon \leq c(x,y)$. Hence $\Theta(\tilde{u})=0$.
	\end{enumerate}
	Let's apply the Fenchel-Rockafellar theorem. The left-hand side of~\eqref{eq:min-max} is
	\begin{align*}
	\inf_{u\in E}\left[\Theta(u)+\Xi(u)\right] &= \inf_{(f,g)\in \calF \times \calG}\{-\int f\ud \mu - \int g\ud \nu;\quad f(x)+g(y)\leq c(x,y)\}\\
	&=-\sup_{(f,g)\in (\calF \times \calG)\cap \calC(c)}~J_{\mu,\nu}(f,g)
	\end{align*}
	Next, we compute the Legendre-Fenchel transform of $\Theta$ and $\Xi$. For any $\pi \in E^\star=M(\calX \times \calY)$ 
	\begin{align*}
	\Theta^\star(\pi) &= \sup_{u\in E}~[\int_{\calX\times \calY} u(x,y)\ud \pi(x,y)-\Theta(u)]\\ &= \sup_{u\in E}~[\int_{\calX\times \calY} u(x,y)\ud \pi(x,y);\quad u(x,y)\leq c(x,y)]\\
	&= \begin{cases}
	\int_{\calX\times \calY} c(x,y)\ud \pi(x,y),\quad \text{if } \pi \in M_+(\calX\times \calY)\\
	+\infty\quad\text{else}
	\end{cases}
	\end{align*}
	where $M_+(\calX \times \calY)$ is the set of non-negative measures on $\calX\times \calY$. The last equality holds because, if $\pi$ is not non-negative, there exists a non-positive function $v\in E$ such that $\int v \ud \pi>0$. Then choosing $u=\lambda v$ with $\lambda \to \infty$ shows that the supremum is $+\infty$. If $\pi$ is non-negative, then clearly the supremum is equal to $\int c \ud \pi$. Let's compute the Legendre-Fenchel transform of $\Xi$. For any $\pi \in E^\star=M(\calX \times \calY)$ 
	\begin{align*}
	\Xi^\star(-\pi) &= \sup_{u\in E} [-\int_{\calX\times \calY} u(x,y)\ud \pi(x,y)-\Xi(u)]\\&=\sup_{(f,g) \in \calF \times \calG} [-\int_{\calX\times \calY} (f(x)+g(y))\ud \pi(x,y) + \int_\calX f(x) \ud\mu(x) + \int_\calY g(y) \ud \nu(y)]\\
	&=\begin{cases}
	0,\quad \text{if } \pi \in \Pi_\succeq^\calK(\mu,\nu) \\
	+\infty\quad\text{else}
	\end{cases}
	\end{align*} 
	The last equality holds because
	\begin{enumerate}
		\item If $\pi \in \Pi_\succeq^\calK(\mu,\nu)$, then we have 
		\begin{align*}
		\int f(x)\ud \pi(x,y&\geq \int f(x)\ud \mu(x),\quad \forall f \in \calF,\\ \int g(y)\ud \pi(x,y)&\geq\int g(y)\ud \nu(y),\quad \forall g \in \calG,
		\end{align*}
		Therefore, inside the supremum is smaller than zero and the supremum is achieved with $f=g=0$. (By definition, $0$ is contained in a cone)
		\item Else if $\pi \notin \Pi_\succeq^\calK(\mu,\nu)$, then there exists $\tilde{f}\in \calF$ (or similarly for some $\tilde{g} \in \calG$) such that $\int \tilde{f}(x)\ud\pi(x,y)<\int \tilde{f}(x)\ud \mu(x)$. Then with the choice $f = \lambda \tilde{f} $ with $\lambda \to +\infty$ the supremum is $+\infty$ (Here the assumption that $\calF$ is a cone is used)
	\end{enumerate} 
	Therefore, the right-hand side of~\eqref{eq:min-max} is:
	\begin{align*}
	\max_{\pi \in E^\star}\left[-\Theta^\star(\pi)-\Xi^\star(-\pi)\right] 
	&= - \min_{\pi\in E^\star}\{\int c(x,y)\ud \pi(x,y);\quad \pi \in M_+(\calX \times \calY)\cap \Pi_\succeq^\calK(\mu,\nu) \}\\
	&= - \min_{\pi\in E^\star}\{I_c(\pi);\quad \pi \in M_+(\calX \times \calY)\cap \Pi_\succeq^\calK(\mu,\nu) \}
	\end{align*}
	Putting everything together and changing signs concludes the proof.
\end{proof}	
\subsection{Proof of Theorem~\ref{thm:Brenier-restricted}}\label{apdx:proof-Brenier-restricted}
\begin{proof}
This is a modification of the proof of \cite[Theorem 2.12]{villani2003topics}.
	\begin{enumerate}
		\item Suppose there exists $(f,g)$ and $\pi$ such that~\eqref{eq:optimality-pi-contrained} is true. Then, $y \in \partial f^{\star\star}(x)$ implies $f^{\star\star}(x) + f^{\star}(y)=\langle x, y\rangle$ for all $(x,y) \in \op{supp}(\pi)$. And $f^\star(y)=g(y)$ implies $f^{\star\star}(x) + g(y)=\langle x, y\rangle$ for all $(x,y) \in \op{supp}(\pi)$. Also because of the constraint $f(x)+g(y) \geq \langle x, y\rangle$ we have $f(x)\geq g^\star(x)$. By definition, $f^{\star\star}(x)$ is the largest convex function below $f(x)$. Therefore, $f(x)\geq f^{\star\star}(x)\geq g^\star(x)$. The condition $f(x)=g^\star(x)$ implies $f(x)=f^{\star\star}(x)= g^\star(x)$ for all $x \in \op{supp}(\pi_x)$. Therefore, $f(x)+g(y) = \langle x, y\rangle$ for all $(x,y) \in \op{supp}(\pi)$. Then,
		\begin{align*}
		\int f(x)\ud \mu(x) + \int g(y) \ud \nu(y) = \int (f(x) + g(y)) \ud \pi(x,y) = \int (\langle x, y\rangle)\ud \pi(x,y)
		\end{align*} 
		Therefore, the gap between objective functions of~\eqref{eq:primal-quad-constrained}-\eqref{eq:dual-quad-constrained} is zero. Hence, $\pi$ and $(f,g)$ are optimal.
		
		For the other direction, assume $\pi$ is optimal for~\eqref{eq:primal-quad-constrained}. By assumption, there exists a minimizer $(f,g)$ for~\eqref{eq:dual-quad-constrained}. Then the gap is zero. 
		\begin{align*}
		\int \langle x, y\rangle \ud \pi(x,y) = \int f(x)\ud \mu(x) + \int g(y) \ud \nu(y) \geq \int (f(x)+g(y)) \ud \pi(x,y)
		\end{align*} 
		where the inequality follows because~$\pi \in \Pi_\preceq ([\mu]_\calF,[\nu]_\calG)$. Because of the constraint $(f,g)\in \overline{\calC}$ we have the inequality in other direction, 
		\begin{equation*}
		\int \langle x, y\rangle \ud \pi(x,y) \leq \int (f(x)+g(y)) \ud \pi(x,y)
		\end{equation*}
		Therefore, 
		\begin{align*}
		\int \langle x, y\rangle \ud \pi(x,y) = \int f(x)\ud \mu(x) + \int g(y) \ud \nu(y) = \int (f(x)+g(y)) \ud \pi(x,y)
		\end{align*} 
		and
		\begin{equation*}
		f(x) + g(y)=\langle x, y\rangle,\quad \forall(x,y)\in \op{supp}(\pi)
		\end{equation*}
		Because of the constraint $(f,g)\in \overline{\calC}$ we have $f(x)\geq \sup_{y} (\langle x, y\rangle - g(y))=g^\star(x)$. Therefore, $f(x)=g^\star(x)$ for all $x \in \op{supp}(\pi_x)$. Similarly, $g(y)=f^\star(y)$ for all $y \in \op{supp}(\pi_y)$. Finally the inequality $f(x)\geq f^{\star\star}(x)\geq g^\star(x)$ and the equality $f(x)=g^\star(x)$ for all $x \in\op{supp}(\pi_x)$ imply $f(x)=f^{\star\star}(x)=g^\star(x)$ for all $x \in \op{supp}(\pi_x)$. Similarly $ g(y)=g^{\star\star}(y)=f^\star(y)$ for all $y \in \op{supp}(\pi_y)$. Therefore
		\begin{equation*}
		f^{\star\star}(x) + g^{\star\star}(y)=\langle x, y\rangle,\quad \forall(x,y)\in \op{supp}(\pi)
		\end{equation*}
		It follows that $y \in \partial f^{\star\star}(x)$ and $x \in \partial g^{\star\star}(y)$ for all $(x,y)\in \op{supp}(\pi)$.	
		\item Let $(f,g)$ be an optimal pair. Replace it with $(f^{\star\star},g^{\star\star})$. It is still admissible. Because, for any admissible pair $f(x)\geq \sup_{y}(\langle x, y\rangle - g(y))=g^\star(x)$. Therefore, $f^{\star\star}(x)\geq g^\star(x)$ because $f^{\star\star}(x)$ is the largest convex function below $f(x)$ and $g^\star(x)$ is convex. Therefore, $(f^{\star\star},g)$ is admissible. Similarly, $g(y)\geq f^{\star\star\star}(y)=f^{\star}(y)$ implies $g^{\star\star}(y)\geq f^\star(y)$. Therefore, $(f^{\star\star},g^{\star\star})$ is admissible. They attain a smaller value compared to $(f,g)$ because $f^{\star\star}(x)\leq f(x)$ and $g^{\star\star}(y)\leq g(y)$. 	Therefore, the optimal pair should be of the form $(f^{\star\star},g^{\star\star})$. Hence they are convex.
		\item This is a special case of part (i) and (ii). Because $\calF^{\star\star}\subset \calF$ and $\calG= \calG^{\star\star}=C_b(\calY)$, then the optimal pair is convex, and because $f^\star \in \calG$, the optimal pair is of the form $(f,f^\star)$ 
		Also because $(f,f^\star)$ are bounded on the compact set $\calX$, they are differentiable almost everywhere. 
		The constraint $\pi \in \Pi_{\preceq}([\mu]_\calF,[\nu]_\calG)$ and $\calG = C_b(\calX)$ imply that the marginal $\pi_y = \nu$. The condition for the other marginal $[\pi_x]_\calF \leq [\nu]_\calF$ imply
		\begin{align*}
		\int \tilde{f}(\nabla f^\star(y))\ud \nu(y)=\int \tilde{f}(x)\ud \pi(x,y)=\int \tilde{f}(x)\ud \pi_x(x) \leq \int \tilde{f}(x)\ud \mu(x),\quad \forall f \in \calF
		\end{align*}
	\end{enumerate}
\end{proof}

\subsection{Special Case: Conic Subsets of the Set of Convex Potentials}
The objective of this section is to build a connection from the analysis of duality under general restrictions that was studied so far in~\autoref{sec:restr-orig} to the computational framework proposed in the paper. In order to do so, we consider the optimal transportation problem with quadratic cost as in~\autoref{apdx:conic-qaudratic}. We let $\calF$ to be a subset of convex functions on $\calX$ that also forms a convex cone. And we choose $\calG = C_b(\calY)$ to be the set of bounded continuous functions. In this special setting,~\autoref{thm:duality}-(ii) applies and and we may express the dual problem~\eqref{eq:dual-quad-constrained} as
\begin{equation*}
\inf_{(f,g) \in \bar{\calC}\cap \calK}~ \int f(x)\ud \mu(x) + \int g(y)\ud \nu(y) = \inf_{f \in \calF}~ \int f(x)\ud \mu(x) + \int f^\star(y)\ud \nu(y)
\end{equation*}
This is important from a computational standpoint, as satisfying the constraint $\bar{\calC}$ in general is challenging. 

Also, in this setting, the primal problem is equivalent to:
\begin{align*}
\sup_{\pi \in \Pi^\calK_\preceq(\mu,\nu)} \int \langle x,y \rangle\ud \pi(x,y)&= \sup_{\lambda \preceq_{\calF}\mu } ~\sup_{\pi \in \Pi(\lambda,\nu)} \int \langle x,y \rangle \ud \pi(x,y)
\end{align*}
where we used $ \Pi_\preceq^\calK(\mu,\nu) =\bigcup \bigl\{ \Pi(\tilde{\mu},\tilde{\nu}) :~ \tilde{\mu}\preceq_\calF \mu,\,\tilde{\nu}\preceq_{C_b(\calY)} \nu \bigr\}$ from \autoref{lem:Pi-union} and $\tilde{\nu}\preceq_{C_b(\calY)} \nu ~\Leftrightarrow~\tilde{\nu} = \nu$. Then by strong duality form~\autoref{thm:restricted-duality}
\begin{align*}
\inf_{f \in \calF }~\int f(x) \ud \mu(x) + \int f^\star (y)\ud \nu(y)= \sup_{\lambda \preceq_{\calF}\mu } ~\sup_{\pi \in \Pi(\lambda,\nu)} \int \langle x,y \rangle \ud \pi(x,y)
\end{align*} 
This is the basic result in this setting, that leads to strong conclusions about the properties of the approximate metric as summarized in~\autoref{prop:conic-all}. 

\subsection{Proof of Theorem \ref{prop:conic-all}}\label{app:proof-prop:conic-all}
\begin{enumerate}
	\item Consider the primal and dual problem~\eqref{eq:primal-quad-constrained}-\eqref{eq:dual-quad-constrained} with $\calK = \calF \times C_b(\calX)$ where $\calX$ is the support of $\nu$ which is compact by \autoref{assump}. Then the primal problem is equivalent to:
		\begin{align*}
	\sup_{\pi \in \Pi^\calK_\preceq(\mu,\nu)} \int \langle x,y \rangle\ud \pi(x,y)&= \sup_{\lambda \preceq_{\calF}\mu } ~\sup_{\pi \in \Pi(\lambda,\nu)} \int \langle x,y \rangle \ud \pi(x,y)
	\end{align*}
	where we used \autoref{lem:Pi-union}. The dual problem is equivalent to
	\begin{align*}
\inf_{(f,g) \in \bar{\calC} \cap \calK}~\int f(x) \ud \mu(x) + \int g (y)\ud \nu(y) &= \inf_{f \in \calF }~\int f(x) \ud \mu(x) + \int f^\star (y)\ud \nu(y)
\end{align*}
where we used the fact that optimal $g$ is equal to $f^\star$ according to~\autoref{thm:Brenier-restricted}-(ii). Then by duality form~\autoref{thm:restricted-duality}
\begin{align*}
\inf_{f \in \calF }~\int f(x) \ud \mu(x) + \int f^\star (y)\ud \nu(y)= \sup_{\lambda \preceq_{\calF}\mu } ~\sup_{\pi \in \Pi(\lambda,\nu)} \int \langle x,y \rangle \ud \pi(x,y)
\end{align*} 
Multiplying both sides by $-1$ and adding $ \frac{1}{2}\int \norm{x}_2^2 \ud \mu(x) + \frac{1}{2}\int \norm{y}_2^2 \ud \nu(y) $ yields:
\begin{align*}
 \W_{2,\calF}^2(\mu,\nu) &= \inf_{\lambda \preceq_{\calF}\mu } ~\inf_{\pi \in \Pi(\lambda,\nu)}\left(\frac{1}{2}\int \norm{x}_2^2 \ud \mu(x) + \frac{1}{2}\int \norm{y}_2^2 \ud \nu(y) - \int \langle x,y \rangle \ud \pi(x,y) \right)\\
 &=\inf_{\lambda \preceq_{\calF}\mu } ~\inf_{\pi \in \Pi(\lambda,\nu)}\left(\frac{1}{2}\int \norm{x}_2^2 \ud \mu(x) - \frac{1}{2}\int \norm{x}_2^2 \ud \lambda (y) + \frac{1}{2}\int \norm{x-y}_2^2\ud \pi(x,y) \right)\\
& = \inf_{\lambda \preceq_{\calF}\mu } \left(\frac{1}{2}\int \norm{x}_2^2 \ud \mu(x) - \frac{1}{2}\int \norm{x}_2^2 \ud \lambda (x) + \W^2_2(\lambda,\nu) \right)
\end{align*} 
\item Suppose the right-hand side is true. Then the left-hand side follows from~\eqref{eq:W2F-duality} by choosing $\lambda = \nu$. Now assume the left-hand side is true. Then according to~\eqref{eq:W2F-duality} 
 \begin{equation*}
 \inf_{\lambda \preceq_{\calF}\mu } \left[\left(\frac{1}{2}\int \norm{x}_2^2 \ud \mu(x) - \frac{1}{2}\int \norm{x}_2^2 \ud \lambda (x)\right) + \W^2_2(\lambda,\nu) \right] = 0
 \end{equation*}
Both terms are non-negative. Therefore, it follows that both should be zero at optimality. Therefore, $\exists \lambda \preceq_\calF \mu$ such that $\W_2(\lambda,\nu)=0$. Hence $\lambda = \nu$. As a result $\nu \preceq_\calF \mu$ and $\frac{1}{2}\int \norm{x}_2^2 \ud \mu(x) = \frac{1}{2}\int \norm{x}_2^2 \ud \nu (x)$.

\item By definition, for all $\lambda \in \nabla \calF \# \mu$ there exists $f \in \calF$ and a measurable map $T$ such that $T(x) \in \partial f(x)$ and $\lambda = T\#\mu$. Then consider the joint distribution $\ud \pi(x,y) = \ud \mu(x)\delta_{y=T(x)}$. The marginals of $\pi$ are equal to $\mu$ and $\lambda$. Also for all $(x,y) \in \text{supp}(\pi)$ we have $y=T(x)\in \partial f(x)$. 
Therefore, by~\autoref{thm:Brenier}, $\pi$ is the optimal coupling between $\mu$ and $\lambda$ and $f$ is the optimal potential function that minimizes the dual problem. Because $f \in \calF$, the restriction to $\calF$ does not change the value of the not-restricted dual problem. Therefore, 
\begin{equation}\label{W2F-W2-identity}
\W_2(\mu,\lambda) = \W_{2,\calF}(\mu,\lambda)
\end{equation}
for all $\lambda \in \nabla \calF \# \mu$. Hence for all $\lambda \in \nabla \calF \# \mu$,
\begin{align*}
\W_2(\mu,\nu) &\leq \W_2(\mu,\lambda) + \W_2(\lambda,\nu) \\
&= \W_{2,\calF}(\mu,\lambda) + \W_2(\lambda,\nu)\\
&= \inf_{\tilde{\mu} \preceq_\calF \mu} ~\left[\W_2^2(\tilde{\mu},\lambda) + \int \frac{1}{2}\norm{x}_2^2 \ud \mu(x) -\int \frac{1}{2}\norm{x}_2^2 \ud \tilde{\mu}(x) \right]^{1/2} + \W_2(\lambda,\nu)\\
&\leq \inf_{\tilde{\mu} \preceq_\calF \mu} ~\left[2\W_2^2(\tilde{\mu},\nu) + 2\W_2^2(\lambda,\nu)+ \int \frac{1}{2}\norm{x}_2^2 \ud \mu(x) -\int \frac{1}{2}\norm{x}_2^2 \ud \tilde{\mu}(x) \right]^{1/2} + \W_2(\lambda,\nu)\\
&\leq ~\left[2\W^2_{2,\calF}({\mu},\nu) + 2\W_2^2(\lambda,\nu) \right]^{1/2} + \W_2^2(\lambda,\nu)
\end{align*}
where in the second line we used\eqref{W2F-W2-identity}, and on the third and last line we used~\eqref{eq:W2F-duality}. Letting $\lambda = \text{Proj}(\nu,\nabla \calF \# \mu)$ concludes the result. 
\end{enumerate}

\subsection{What's Next? Restricting the Reduced Dual Form}\label{sec:bridge}

The bulk of \autoref{sec:restr-orig} is concerned with duality in infinite-dimensional linear programming and the implications for optimal transports. In other words, we are restricting the original dual Kantorovich problem \eqref{eq:dual-form} to a cone $\calF$ to get \eqref{eq:dual-constrained} or \eqref{eq:dual-quad-constrained}. However, these are {\em infinite-dimensional constrained optimization problems} and are hard to solve in practice. Then, according to \autoref{thm:Brenier-restricted}, we know that the optimal solution pair $(f,g)$ are conjugate to each other and $f$ lies in $\calF \cap \cvx(\calX)$. The first statement allows for turning the problem into an {\em equivalent} {\em unconstrained} form which opens the door to many more optimization algorithms; this helps us improve on the computational aspect. Note that regularization-based approaches in optimal transport also turn the problem into an unconstrained form but they do so in an inexact way which introduces {\em bias}; see \autoref{sec:prior}. The second implication of \autoref{thm:Brenier-restricted} suggests that we can optimize over $\calF \cap \cvx(\calX) \subseteq \cvx(\calX)$. However, given $\calF$ it is in general not easy to have a computational characterization for $\calF \cap \cvx(\calX)$. Therefore, the intersection may not be expressive enough for our purposes; e.g., while the class of polynomials of certain degree is big enough for many purposes, the subset of such polynomials that are convex is much smaller. Therefore, we propose to consider restriction to sets $\calF \subseteq \cvx(\calX)$ from the beginning (\autoref{sec:restr-cvx}). This way, we have a control on the possible optimal solutions, hence on the overall behavior of the optimal value function, i.e., the distance. This is how we get a better handle on the generalization (statistical) aspects. Some of such sets will be cones (as in \autoref{sec:restr-finite}, some cases in \autoref{sec:restr-quad}), and some will not (as in some cases in \autoref{sec:restr-quad}, \autoref{sec:restr-PLQ}, \autoref{sec:restr-ICNN}). For the former cases we can use the general results of \autoref{sec:restr-orig}. However, for the latter cases we resort to a case by case analysis besides the unified results in \autoref{sec:restr-cvx}.

\section{Further Results for Some Parametrized Subsets of Convex Functions}\label{sec:proposed-restr}

In this section, we consider several class of convex functions and study their theoretical properties in the context of~\autoref{sec:main-results}. Hand-picking the restriction class allows for adapting to the requirement of the problem at hand. For example, we may be interested in learning a probability distribution that only matches certain moments of the underlying distribution from which we have samples. In such case, using an appropriate approximate metric allows for convergence with fewer samples and at a lower computational cost.

\subsection{A Finitely Generated Set of Convex Functions}\label{sec:restr-finite}
Consider $f_0:x\mapsto \frac{1}{2}\norm{x}_2^2$ as well as $M$ closed convex functions $f_1, \ldots, f_M \in \cvx(\real^d)$. Define 
\[
\calF_0 = \bigl\{f_0, f_1, \ldots, f_M \bigr\}\subset \cvx(\real^n) \quad,\quad \calF = \op{cone}(\calF_0)
\]
where $\calF$ is the convex conic hull of $\calF_0$. Given a measure $\mu$ define the $M$-dimensional vector of its moments with respect to~$\calF_0$ as $m_{\calF_0}(\mu) \coloneqq \begin{bmatrix} m_{f_0}(\mu) , m_{f_1}(\mu) , \cdots , m_{f_M}(\mu) \end{bmatrix} \in\real^{M}$ where $m_{f}(\mu)\coloneqq \int f(x) \ud \mu(x)$. 

\paragraph{Moment-matching.}
According to the part-(ii) of the \autoref{prop:conic-all}, we have
\begin{equation*}
\W_{2,\calF}(\mu,\nu) = 0\quad \iff\quad 
\mu \succeq_{\text{cone}(\calF_0)} \nu \quad \iff\quad 
\mu \succeq_{\calF_0} \nu \quad \iff\quad 
m_{\calF_0}(\mu) \geq m_{\calF_0}(\nu) 
\end{equation*}
where the last inequality is entry-wise.

\paragraph{Approximability.} 
Let $X$ be a random variable whose probability distribution is equal to $\mu$. Then, $\nabla \calF \# \mu$ consists of all distributions corresponding to random variables which belong to the set 
\[\nabla \calF (X) = \left\{Y = \sum_{m=0}^M \alpha_m\nabla f_m(X):~\alpha_m \geq 0,~m \in [M]\right\}.\] As a result, the approximate metric between $X$ and any $Y \in \nabla \calF (X)$ is exact.

\paragraph{Transport map.}
For $\theta \in \real^M_+$, let $f(x;\theta) = \sum_{m=1}^M \theta_m f_m(x)$. Then, $\frac{\partial f}{\partial \theta_m} (x;\theta)= f_m(x)$ for all $m\in [M]$. Therefore, the tangent space as defined in~\eqref{eq:Tan-space} is equal to 
\begin{equation*}
\Tan_\theta \calF = \text{span}(\calF_0).
\end{equation*} 
Now consider the approximate optimal transport map $T_{\calF}(x)=\nabla f^\star(x;\bar{\theta})$ as defined in~\eqref{eq:approx-opt-map}. Then, as a result of~\autoref{prop:approximate-map}, if $\bar{\theta}$ belongs to the interior of $\real^M_{+}$ (i.e., its components are strictly positive), we have $\int g(x)\ud \mu(x) = \int g(T_{\calF}(y))\ud \nu(y)$ for all $g \in \Tan_{\bar{\theta}}\calF = \text{span}(\calF_0)$. This implies \[m_{\calF_0}(\mu)=m_{\calF_0}(T_\calF \# \nu).\]

\subsection{Convex Quadratic Functions}\label{sec:restr-quad}
Consider subsets of the set of convex quadratic functions parametrized as
\begin{equation*}
\Q(\Theta)\coloneqq \bigl\{f:x\mapsto \frac{1}{2}x^\top A x + b^\top x;~(A,b)\in \Theta \bigr\} \subset \cvx(\mathbb{R}^n)
\end{equation*}
where $\Theta \subseteq \mathbb{S}^n_{++}\times \real^n$. For any quadratic function $f \in \Q(\Theta)$, namely $f:x\mapsto \frac{1}{2}x^\top A x + b^\top x$ for $\theta=(A,b)\in\Theta$, the convex conjugate can be expressed as $f^\star: y\mapsto \frac{1}{2}(y-b)^\top A^{-1}(y-b)$. We will use this restriction class to illustrate the theoretical results in~\autoref{sec:main-results}. The class of quadratic functions is also studied in the context of GAN by \citet{feizi2017understanding}.

\paragraph{Moment-matching.}
From the above, with $\theta=(A,b)$, we have 
$\frac{\partial f}{\partial A_{ij}}(x;\theta)=x_ix_j$ and $\frac{\partial f}{\partial b_{i}}(x;\theta)=x_i$ for all $i,j\in[n]$. 
Therefore, the tangent space defined in \eqref{eq:Tan-space} is given by 
\begin{equation}\label{eq:quad-T}
\Tan_{\theta} \calF = \op{span}\{x_ix_j:~ i,j=0,1,\ldots,n\} \supset \Q(\Theta)
\end{equation}
for all $\theta \in \Theta$, where we define $x_0 = 1$. Note that, in this special case, the tangent space does not depend on $\theta$. 

We have $f(x;\theta_0) = \frac{1}{2}\norm{x}_2^2$ if and only if $\theta_0=(I_{d\times d},0_{d\times 1})$. Therefore, the set $\Theta_0=\{\theta_0\}$ is non-empty. Then, according to \autoref{prop:W2F-metric}, if $\theta_0$ is in the interior of $\Theta$ the moment matching property is satisfied for all functions that belong to the tangent space $\Tan_{\theta_0}\calF$ (given in \eqref{eq:quad-T}). Hence, if the approximate metric, defined with respect to the class of convex quadratic functions, is zero, then the first and second moments are equal.

\paragraph{Approximability.}
For $f$ in $\Q(\Theta)$, we have $\nabla f(x) = Ax + b$. This is an affine transformation. Therefore, according to the \autoref{prop:W2F-approximability}, this convex quadratic function class can exactly approximate the $\W_2$ distance between any two distributions that are related to each other with an affine transformation.

\paragraph{Metric properties.} Consider the case $\Theta =S^n_{++} \times \real^n $. In this case, $\mathcal{Q}(\Theta)$ forms a convex cone. Therefore, one can prove strong results about the metric properties of the approximate metric. 
Define the map $G:\mathcal{P}_{2,+}(\real^n) \to \mathcal{P}_{2,+}(\real^n)$ such that it takes a probability distribution $\mu$ and outputs a Gaussian distribution with the same mean and covariance, 
\begin{equation*}
G(\mu) = \mathcal{N}( m_1(\mu), m_2(\mu) ).
\end{equation*}

\begin{proposition}\label{prop:Quad-metric}
	Consider $\Theta = \mathbb{S}^n_{++}\times \real^n$. 
	Consider the $L^2$-Wasserstein distance restricted to the class of all convex quadratic functions. Then,
	\begin{enumerate}
		\item For all $\mu,\nu \in \mathcal{P}_{2,+}(\real^n)$ 
		\begin{equation}\label{eq:W2-quad-identity}
		\W_{2,\Q(\Theta)}(\mu,\nu) = \W_{2}(G(\mu),G(\nu))
		\end{equation}
		\item $\W_{2,\Q(\Theta)}$ is a pseudo-metric on the space of probability distributions~$\mathcal{P}_{2,+}(\real^n)$. 
		\item $\W_{2,\Q(\Theta)}$ is a metric on the space of Gaussian distributions with a positive definite covariance matrix.
	\end{enumerate}
\end{proposition}
\begin{proof}
	\begin{enumerate}
		\item 
		In the case of quadratic functions, it is easy to see that $\tilde J_{\mu,\nu}(\theta) = \tilde J_{G(\mu),G(\nu)}(\theta)$ for all $\theta \in \Theta$; 
		the value of the objective function does not change if one replaces $\mu$ and $\nu$ with other distributions with the same mean and covariance, because the value depends only on the mean and the covariance. Therefore, 
		\begin{equation*}
		\inf_{\theta\in \Theta}~ \tilde J_{\mu,\nu}(f) = \inf_{\theta\in\Theta}~\tilde J_{G(\mu),G(\nu)}(f).
		\end{equation*}
		Since any two Gaussian distributions can be mapped to each other using an affine transformation, an optimal pair of functions in computing $\W_{2}(G(\mu),G(\nu))$ is going to be a quadratic function. Therefore, the righ-hand side in the above corresponds to $\W_{2}(G(\mu),G(\nu))$. 
		This establishes the claim. 
		
		\item From the identity~\eqref{eq:W2-quad-identity}, one can easily conclude the three properties of the pseudo-metric: $\forall \mu,\nu,\lambda \in \mathcal{P}_2(\real^n)$ 
		\begin{align*}
		\text{(i)}\quad \W_{2,\Q}(\mu,\mu)&=\W_2(G(\mu),G(\mu))= 0\\
		\text{(ii)}\quad \W_{2,\Q}(\mu,\nu) &= \W_2(G(\mu),G(\nu)) = \W_2(G(\nu),G(\mu))=\W_{2,\Q}(\nu,\mu)\\
		\text{(iii)}\quad \W_{2,\Q}(\mu,\nu) &= \W_2(G(\mu),G(\nu)) \leq \W_2(G(\mu),G(\lambda)) + \W_2(G(\lambda),G(\nu)) \\&= \W_{2,\Q}(\mu,\lambda) + \W_{2,\Q}(\lambda,\nu)
		\end{align*}
		\item On the space of Gaussian probability distributions, the map $G$ is an identity map. Hence for all Gaussian distributions $0 = \W_{2,\Q(\Theta)}(\mu,\nu) = \W_2(\mu,\nu)$ implies $\mu=\nu$. This, together with the pseudo-metric property, establishes the claim. 
		
	\end{enumerate}
\end{proof}

\paragraph{Transport map.}
Observe that $\nabla_A f(x;A,b) = \frac{1}{2}x x^\top$ and $\nabla_b f(x;A,b) = x$. As a result, according to \autoref{prop:approximate-map}, and the independence of $\Tan_\theta \calF$ from $\theta$ (discussed in the beginning of this secgtion), the transport map matches the means and the covariances, namely 
\begin{align*}
\int x x^\top \ud \mu(x) = \int T_\calF(y)T_\calF(y)^\top \ud \nu(y) \quad,\quad 
\int x \ud \mu(x) = \int T_\calF(y)\ud \nu(y)
\end{align*}

\paragraph{The derivative.} 
The objective function $\tilde{J}$ defined in~\eqref{eq:min-theta} evaluted for the class of convex quadratic functions is given by 
\begin{align}\label{eq:J-quad}
\tilde{J}_{\mu,\nu}(A,b) = \int (\frac{1}{2}xAx^\top +b^\top x)\ud \mu(x) + \int \frac{1}{2}(y-b)^\top A^{-1}(y-b)\ud \nu(y).
\end{align}
Then the derivatives with respect to $A$ and $b$ are given by:
\begin{align*}
\nabla_A \tilde{J}_{\mu,\nu}(A,b) &= \frac{1}{2} \int x x^\top \ud \mu(x) - \frac{1}{2}\int A^{-1}(y-b)(y-b)^\top A^{-1} \ud \nu(y)\\
\nabla_b \tilde{J}_{\mu,\nu}(A,b) &= \int x \ud \mu(x) - \int A^{-1}(y-b) \ud \nu(y)
\end{align*}
The same result can be seen from \autoref{prop:J-derivative}.

\paragraph{Optimization landscape.} 
In this special setting, one can analyze the optimization landscape of the optimization problem 
\begin{equation}\label{eq:opt-problem-quadratic}
\inf_{(A,b) \in \Theta}~\int f(x;A,b)\ud \mu(x) + \int f^\star(y;A,b)\ud \nu(y).
\end{equation}
Let $\tilde{J}_{\mu,\nu}(A,b)$ denote the value of the objective function. 
Understanding the landscape for optimization problem helps in devising appropriate algorithms for computing the approximations. 

\begin{proposition} \label{prop:Quad-landscape}
	Consider $\mu,\nu \in \mathcal{P}_{2,+}(\real^n)$ and $X\sim \mu, Y\sim \nu$. 
	Consider the optimization problem~\eqref{eq:opt-problem-quadratic}. 
	The objective function $\tilde{J}_{\mu,\nu}$ is convex in $(A,b)$ on the domain $\Theta=\mathbb{S}^n_{++} \times \real^n$. There is a unique minimizer $(\bar{A},\bar{b})$ given by
	\begin{equation*}
	\bar{A} = \Sigma_X^{-\frac{1}{2}}\left(\Sigma_X^{\frac{1}{2}}\Sigma_Y\Sigma_X^{\frac{1}{2}}\right)^{\frac{1}{2}}\Sigma_X^{-\frac{1}{2}},\quad \bar{b} = m_Y - \bar{A}m_X,\quad c \in \real
	\end{equation*}	
	and the optimal value is
	\begin{equation*}
	\min_{A,b}~\tilde{J}_{\mu,\nu}(A,b)=m_X^\top m_Y + \trace((\Sigma_X^{1/2}\Sigma_Y\Sigma_X^{1/2})^{1/2}).
	\end{equation*}
\end{proposition}
\begin{proof}
	The first two terms of \eqref{eq:J-quad} are convex because they are linear in $A$ and $b$. It remains to show that the last term is also convex. We show this by establishing the convexity of its epigraph. Note that for all $(y,t)\in\real^n\times \real$,
	\[
	(y-b)^\top A^{-1}(y-b) \leq t ~~\Leftrightarrow~~ \begin{bmatrix} A &(y-b)^\top\\y-b & t\end{bmatrix}\succeq 0 \text{ and $A$ is invertible}.
	\]
	Therefore, the epigraph $\{(A,b,t) \in \mathbb{S}^n_{++} \times \real^n \times \real;~(y-b)^\top A^{-1}(y-b) \leq t\}$ is convex as the following set is convex
	\begin{align*}
	\{(A,b,t) \in \mathbb{S}^n_{++} \times \real^n \times \real;~ \begin{bmatrix} A &(y-b)^\top\\y-b & t\end{bmatrix}\succeq 0\}
	\end{align*}
	which follows from convexity of the cone of positive semi-definite matrices.
	Alternatively, we can write the objective function as
	\begin{equation*}
	\begin{aligned}
	\tilde{J}_{\mu,\nu}(A,b)&= \int \left[\frac{1}{2}x^\top A x + b^\top x \right] \ud\mu(x)+ \int \sup_z\left( z^\top y - \frac{1}{2}z^\top A z - b^\top z \right) \ud\nu(y).
	\end{aligned}
	\end{equation*} 
	The function inside the supremum is linear in $b$ and $A$. The supremum of linear functions is convex. And the expectation of a convex functions is also convex. 
	
	The rest of the proof follows from \autoref{prop:Quad-metric} and explicit formula of optimal transport map for Gaussian distributions.
\end{proof}

\subsection{Piecewise-Linear-Quadratic Functions}\label{sec:restr-PLQ}
Consider a class of parameterized convex functions of the form \[\calF = \{f:x\mapsto \max_{m\in [M]} (\frac{1}{2}x^\top A_m x + b_m^\top x + c_m);~(A_m,b_m,c_m) \in S^d_{++} \times \real^d \times \real,~ m \in [M] \}.\] 
It is easy to see that these functions are piecewise-linear-quadratic. 
Define sets $S_m \coloneqq \{x \in \real^d;~(\frac{1}{2}x^\top A_m x + b_m^\top x + c_m) = \max_{n\in[M]}(\frac{1}{2}x^\top A_n x + b_n^\top x + c_n)\}$ as the subset of locations where the piece corresponding to the index $m$ attains the maximum. 
\paragraph{Approximability.}
{Let $X$ be a random variable whose probability distribution is equal to $\mu$. For any $x\in\real^d$, define 
\[
\nabla \calF (x)= \op{conv}\left\{y = A_{m}x + b_{m}:~ m\in \op{Argmax}_{m\in[M]}~(\frac{1}{2}x^\top A_m x + b_m^\top x + c_m)\right\}
\]
Then, $\nabla \calF \# \mu$ consists of all distributions corresponding to random variables which belong to the set $\nabla \calF (X)$. As a result, the approximate metric between $X$ and any $Y \in \nabla \calF (X)$ is exact.}

\subsection{Input-Convex Neural Networks}\label{sec:restr-ICNN}
Consider the class of convex functions \[\calF = \{f:x\mapsto w^\top \sigma^2(Ax+b) ;~ (w,A,b)\in \real^{2d}_+ \times \real^{2d\times d} \times \real^{2d}\}\] where $\sigma$ is the ReLU activation. Any function $f\in \calF$ is also expressed as $f(x;\theta) = \sum_{i=1}^{2d} w_i(a_i^\top x+b_i)_+^2$, where $(\alpha)_+= \max\{\alpha,0\}$ and $A^\top = [a_1, \cdots, a_{2d}]$. 

\paragraph{Moment-matching.}
	Observe that $f(x;\theta_0) = \frac{1}{2}\sum_{i=1}^d x_i^2$ for
	\begin{equation*}
	\theta_0=(w,A,b)=( \frac{1}{2} \one_{2d\times 1} ,\begin{bmatrix}
	I_{d\times d} \\ -I_{d\times d} 
	\end{bmatrix}
	,\zero_{2d\times 1}) \in \Theta_0.
	\end{equation*}
	Then, the tangent space $\Tan_{\theta_0}\calF$ is given by functions
	\begin{equation*}\begin{array}{lll}
	\frac{\partial f}{\partial w_i}(x;\theta_0) = x_i^2\mathds{1}_{x_i\geq 0}\,,&
	\frac{\partial f}{\partial b_i}(x;\theta_0) =x_i\mathds{1}_{x_i\geq 0}\,,&
	\frac{\partial f}{\partial A_{ij}}(x;\theta_0) =x_ix_j\mathds{1}_{x_i\geq 0}\,,\\
	\frac{\partial f}{\partial w_{i+d}}(x;\theta_0) = x_i^2\mathds{1}_{x_i\leq 0}\,,&
	\frac{\partial f}{\partial b_{i+d}}(x;\theta_0) =-x_i\mathds{1}_{x_i\leq 0}\,,&
	\frac{\partial f}{\partial A_{i+d,j}}(x;\theta_0) =x_{i+d}x_j\mathds{1}_{x_i\leq 0}\,,
	\end{array}\end{equation*}
	for $i,j=1,\ldots, d$. 
	Therefore, for this class of convex functions, if $\theta_0$ is in the interior of $\Theta$, \autoref{prop:W2F-metric} implies that if the approximate metric is zero, then the expectation of the functions noted above with respect to the two distributions are equal. Note that, these are not all the statistics that are being matched as other members of $\Theta_0$ may provide other statistics. 

\paragraph{Approximability.} Consider the problem of learning a symmetric one-dimensional distribution $\ud \nu = \frac{1}{2}\delta_{\{x=-v\}} + \frac{1}{2}\delta_{\{x=v\}}$ where $v\geq 0$. Suppose the generator generates distributions of the form $\ud \mu(x) = \frac{1}{2}\delta_{\{x=-u\}} + \frac{1}{2}\delta_{\{x=u\}}$ where $u\geq 0$ is the parameter of the generator. The parameter $u$ is learned by minimizing $\W_{2,\calF}(\mu,\nu)$ where the discriminator function class $\calF \coloneqq \{f:x\mapsto \max(\sigma^2(x-w),\sigma^2(-x-w));~|w|\leq L\}$ where $\sigma(x)$ is the ReLU function. 
The derivative of a function $f \in \calF$ is given by:
\begin{align*}
\nabla f(x;w) = \begin{cases} (x-w)\mathds{1}_{x\geq w} + (x+w)\mathds{1}_{x\leq -w} & w>0\\
(x-w)\mathds{1}_{x\geq 0} + (x+w)\mathds{1}_{x\leq 0} & w\leq 0
\end{cases}
\end{align*}
Then 
\begin{align*}
\nabla f \# \mu = \begin{cases}
\frac{1}{2}\delta_{\{x=-u\}} + \frac{1}{2}\delta_{\{x=u\}}& w\geq u \\
\frac{1}{2}\delta_{\{x=-u+w\}} + \frac{1}{2}\delta_{\{x=u-w\}}& w\leq u 
\end{cases}
\end{align*}
Therefore, $\nabla \calF \# \mu$ contains all distributions of the form 
$\ud \nu = \frac{1}{2}\delta_{\{x=-v\}} + \frac{1}{2}\delta_{\{x=v\}}$ for $v\in[0, L+u]$.
As a result of \autoref{prop:W2F-approximability}-\eqref{prop:W2F-approximability-1}, $\W_{2,\calF}(\mu,\nu)=\W_2(\mu,\nu)= |u-v|^2$. 
	
	Furthermore, if $\ud \nu = (\frac{1}{2}-\alpha)\delta_{\{x=-v\}} + (\frac{1}{2}+\alpha)\delta_{\{x=v\}}$ is slightly varied and does not belong to $\nabla \calF \# \mu$ then \autoref{prop:W2F-approximability}-\eqref{prop:W2F-approximability-1} provides an upper-bound for the error, with $\epsilon \leq \W_2(\lambda,\nu)=2\alpha|v|$ where $\ud \lambda =\frac{1}{2}\delta_{\{x=-v\}} + \frac{1}{2}\delta_{\{x=v\}}$. Also $|\nabla f^\star(y;w) -x|\leq w \leq L$. As a result $c_1=0$ and $c_2=L$. Hence $c=L$. 

\end{document}